\documentclass[a4paper,12pt, leqno]{article}
\usepackage{amsmath, amsthm, amssymb, fullpage}
\usepackage[english]{babel}
\usepackage[utf8]{inputenc}
\usepackage{eucal}
\usepackage{thmtools}
\usepackage{float} 
\usepackage{thm-restate}
\usepackage{tikz-cd}
\usepackage{geometry}
\usepackage[colorlinks, bookmarks=true, backref=page]{hyperref}
\hypersetup{
	colorlinks,
	citecolor=blue,
	linkcolor=blue,
	urlcolor=blue}
\usepackage{cleveref}
\usepackage{titlesec}
\usepackage{xcolor}
\usepackage{mathtools}
\usepackage[nottoc]{tocbibind}
\usepackage[noadjust]{cite}
\usepackage{graphicx}
\usepackage{caption}
\usepackage{subcaption}
\usepackage{bm}
\usepackage{mathrsfs}
\usepackage{comment}
\usepackage{array}
\numberwithin{equation}{section}

\newcommand{\gefin}{genuinely filling}
\newcommand{\gefi}{genuinely filling }

\newcommand{\toge }{totally geodesic }

\newcommand{\bgrn}{Bergeron}
\newcommand{\grass}{Grassmann }
\newcommand{\vl}{Vladimir }
\newcommand{\mkv}{Markovi\'c }
\newcommand{\mkvn}{Markovi\'c}
\newcommand{\poi}{$\pi_1$-injective }

\newcommand{\cble}{commensurable }
\newcommand{\cblen}{commensurable}
\newcommand{\cbility}{commensurability }
\newcommand{\tgj}{\widetilde{\gamma_j}}
\newcommand{\tv}{\tilde{v}}
\newcommand{\tx}{\tilde{x}}
\newcommand{\tpj}{\widetilde{p_j}}

\newcommand{\tpaij}{\widetilde{\Pi_j}}
\newcommand{\tvj}{\widetilde{v_j}}
\newcommand{\txj}{\widetilde{x_j}}
\newcommand{\tpai}{\widetilde{\Pi}}
\newcommand{\diam}{\tn{diam}}

\newcommand{\plm}{\planes(M)}
\newcommand{\sltr}{SL_2(\R)}

\newcommand{\pml}{\mathcal{PML}}
\newcommand{\db}{D_{\beta}}
\newcommand{\dbo}{D_{\beta_1}}
\newcommand{\dbt}{D_{\beta_2}}

\newcommand{\kms}{Kahn, Markovi\'c, and Smilga }

\newcommand{\qf}{quasi-Fuchsian }
\newcommand{\eqf}{$(1+\epsilon)$-\qf}
\newcommand{\di}{\,\mathrm{d}}

\newcommand{\gtm}{\mathcal{G}_2M}

\newcommand{\gn}{\mathfrak{g}_0}

\newcommand{\rank}{\tn{rank}}

\newcommand{\gr}{\mathcal{G}_2}
\newcommand{\grm}{\mathcal{G}_2M}
\newcommand{\hp}{\mathbb{H}^2}

\newcommand{\ml}{\mathcal{L}}
\newcommand{\mul}{\mu_{\ml}}
\newcommand{\tsj}{\tilde{S}_j}
\newcommand{\tsi}{\tilde{S}_i}
\newcommand{\ts}{\widetilde{S}}
\newcommand{\psltr}{PSL_2(\R)}
\newcommand{\psltc}{PSL_2(\C)}

\newcommand{\gt}{\mathcal{G}_2}

\newtheorem{theorem}{Theorem}[section]
\newtheorem{cor}[theorem]{Corollary}
\newtheorem{lem}[theorem]{Lemma}
\newtheorem{q}[theorem]{Question}

\theoremstyle{definition} 
\newtheorem{defi}[theorem]{Definition}
\newtheorem{rem}[theorem]{Remark}
\newtheorem{exa}[theorem]{Example}

\newcommand{\agn}{asymptotically geodesic}
\newcommand{\ag}{asymptotically geodesic }
\newcommand{\fhy}{\mathcal{F}\hy}
\newcommand{\genus}{\mathfrak{g}}
\newcommand{\gs}{\mathfrak{g}[S]}
\newcommand{\rk}{\tn{rank}}

\newcommand{\las}{least area }
\newcommand{\stfs}{strongly filling }
\newcommand{\stf}{strongly filling}

\newcommand{\catzeros}{CAT(0) }

\newcommand{\htms}{hyperbolic $3$-manifold }

\newcommand{\km}{Kahn-Markovi\'c }

\newcommand{\tg}{\tilde{\gamma}}
\newcommand{\tp}{\widetilde{P}}
\newcommand{\tm}{\widetilde{M}}

\newcommand{\area}{\textnormal{Area}}

\newcommand{\bd}{\overline{D}}

\newcommand{\C}{\mathbb{C}}
\newcommand{\ds}{\displaystyle}

\newcommand{\hy}{\mathbb{H}^3}
\newcommand{\inj}{\textnormal{inj}}

\newcommand{\liouville}{\mathcal{L}}

\newcommand{\Q}{\mathbb{Q}}
\newcommand{\R}{\mathbb{R}}

\newcommand{\tn}{\textnormal}

\newcommand{\T}{\mathbb{T}}

\newcommand{\vol}{\textnormal{Vol}}
\newcommand{\Z}{\mathbb{Z}}

\newcommand{\pr}{\tn{Pr}}
\newcommand{\planes}{\mathcal{P}}
\newcommand{\isom}{\tn{Isom}}
\newcommand{\hs}{\mathbb{H}^n}
\newcommand{\me}{\mathcal{E}}
\newcommand{\diff}{\,\tn{d}}

\newcommand{\gtxm}{(\gt)_xM}

\newcommand{\im}{\tn{Im}}

\newcommand{\dist}{\textbf{\tn{d}}}

\usepackage{lipsum}
\let\OLDthebibliography\thebibliography
\renewcommand\thebibliography[1]{
	\OLDthebibliography{#1}
	\setlength{\parskip}{1.5pt}
	\setlength{\itemsep}{0pt plus 0.25ex}
}
\newcommand{\address}[1]{%
	\par\noindent\textit{Address: }#1\par\vspace{0.5em}%
}
\newcommand{\email}[1]{%
	\par\noindent\textit{E-mail: }\texttt{#1}\par\vspace{1em}%
}

\title{Nearly geodesic surfaces are filling}
\author{Xiaolong Hans Han}

\date{}
\DeclareUnicodeCharacter{2212}{-}
\begin{document}
	\noindent
	\maketitle
	{\centering\footnotesize \quad Dedicated to Nathan M. Dunfield with gratitude, in celebration of his 50th birthday, for \\ his support since 2015. \par}
	\begin{abstract}
	Let $M$ be a closed hyperbolic $3$-manifold. We consider equivalence class $[S]$ of connected, orientable, closed surfaces immersed in $M$, modulo homotopy and commensurability. A class $[S]$ is \textit{filling} if for any $S' \in [S]$, $M-S'$ consists of components that are contractible in $M$. We prove that there exist $\epsilon_0, \gn>0$ such that excluding finitely many classes of totally geodesic surfaces of genus $<\gn $, all $(1+\epsilon_0)$-\qf $[S]$ are filling. As a corollary, embedded \poi non totally geodesic  \qf surfaces in $M$ have constants bounded below by $1+\epsilon_0$. This also gives a curvature gap theorem for embedded minimal surfaces. Each of these filling surfaces separates any pair of distinct points at $\partial \tm$. Crucial tools include the rigidity results of Mozes-Shah, Ratner, and Shah. This work is inspired by a question of Wu and Xue whether random geodesics on random hyperbolic surfaces are filling. 
	\end{abstract}
 \noindent
\section{Introduction}
\vspace{-0.1in}
Let $M$ be a closed hyperbolic $3$-manifold and $\tm$ the universal cover. 
A \emph{surface} $S$ in $M$ refers to a smooth \poi immersion $S\hookrightarrow M$ of a connected, orientable, closed surface $S$. We denote by $[S]$ an equivalence class of surfaces in $M$ modulo homotopy and commensurability. The genus $\gs$ of $[S]$ is the smallest genus of an orientable surface in $[S]$.  
 A surface $[S]$ is \textit{filling} if for every $S' \in [S]$, $M-S'$ consists of components contractible in $M$ (but these components as topological spaces themselves may not be contractible). Rubinstein-Sageev \cite{rsFillingSurfaces} define $[S]$ to be \textit{strongly filling} if for any $p \neq q \in \partial \tm$, there is an $\ts\subset \tm$ covering $S$ such that $p,q $ belong to different components of $\partial\tm -\partial \ts$. \cite{rsFillingSurfaces} proves that a strongly filling class is filling. 
 Our first theorem is the following.
\begin{theorem}\label{nearly geodesic are filling}
	Let $M$ be a closed hyperbolic $3$-manifold. There exist $\gn, \epsilon_0 >0$ such that excluding finitely many classes of totally geodesic surfaces of genus $<\gn $, all $(1+\epsilon_0)$-\qf surfaces are strongly filling.
\end{theorem}
\noindent For any $\epsilon>0$,	Jeremy Kahn and Vladimir \mkv \cite{kmImmersingAlmostGeodesic} prove that $M$ has infinitely many $(1+\epsilon)$-\qf surfaces. \mkv suggested to the author that most surfaces constructed in \cite{kmCountingEssentialSurfaces} should be filling which help formulate \Cref{nearly geodesic are filling}. If $[S]$ is strongly filling, it has essential intersection with every geodesic in $M$. If $[S]$ is filling, then by taking a least area representative $S$ (whose existence follows from Schoen-Yau \cite{syExistIncomMinTopo} and Sacks-Uhlenbeck \cite{suMinimalImmersionClosedRiemann}), \cite{rsFillingSurfaces} proves that $M-S$ consists of topological balls. 

For the rest of the paper, we set $\gn $ and $\epsilon_0$ as in \Cref{nearly geodesic are filling}. 
	\vspace{-0.05in}
\begin{cor}\label{embedded minimal surfaces of sufficiently small curvature is geodesic}
	Let $M$ be a closed hyperbolic $3$-manifold. The set of embedded \poi  quasi-Fuchsian surfaces in $M$ is decomposed into 
	\vspace{-0.05in}
	\begin{enumerate}
		\item at most finitely many totally geodesic surfaces of genus $<\gn $, and
		\vspace{-0.07in}
		\item surfaces with quasi-Fuchsian constant $\geq 1+\epsilon_0$. 
	\end{enumerate} 
			\vspace{-0.07in}
\end{cor}
We construct infinitely many closed hyperbolic $3$-manifolds each of which contains infinitely many embedded \poi \qf surfaces in \Cref{example of closed hyp 3mf with infinite embedded QF surfaces}. 
A combination of theorems due to Thurston \cite{twGeoTop} and Bonahon \cite{bfBoutsDesVar} proves that a \poi surface $S$ in $M$ is either quasi-Fuchsian or virtually fibered ($S$ lifts to a fiber of a finite cover of $M$ that is a surface bundle over $S^1$). The supremum principal curvature of a fibered minimal surface $F$ is $> 1$ by \cite{ukclosedminSurfaceHyperbolic,hlBeyondAlmostFu} and the Hausdorff dimension of the limit set of $F$ is $2$. Thus 
\Cref{embedded minimal surfaces of sufficiently small curvature is geodesic} implies that any sequence of \poi surfaces in $M$ is eventually self-intersecting or uniformly away from being totally geodesic. 
\begin{cor}\label{gap min surface}
	Let $M$ be a closed hyperbolic $3$-manifold. If $S \hookrightarrow M$ is an embedded minimal surface, then $S$ is either a totally geodesic surface of genus $<\gn $ or has supreme principal curvature $\geq \frac{\epsilon_0}{\epsilon_0+2}$.
\end{cor}
	\begin{figure}[h]
		\centering
		\includegraphics[scale=0.2]{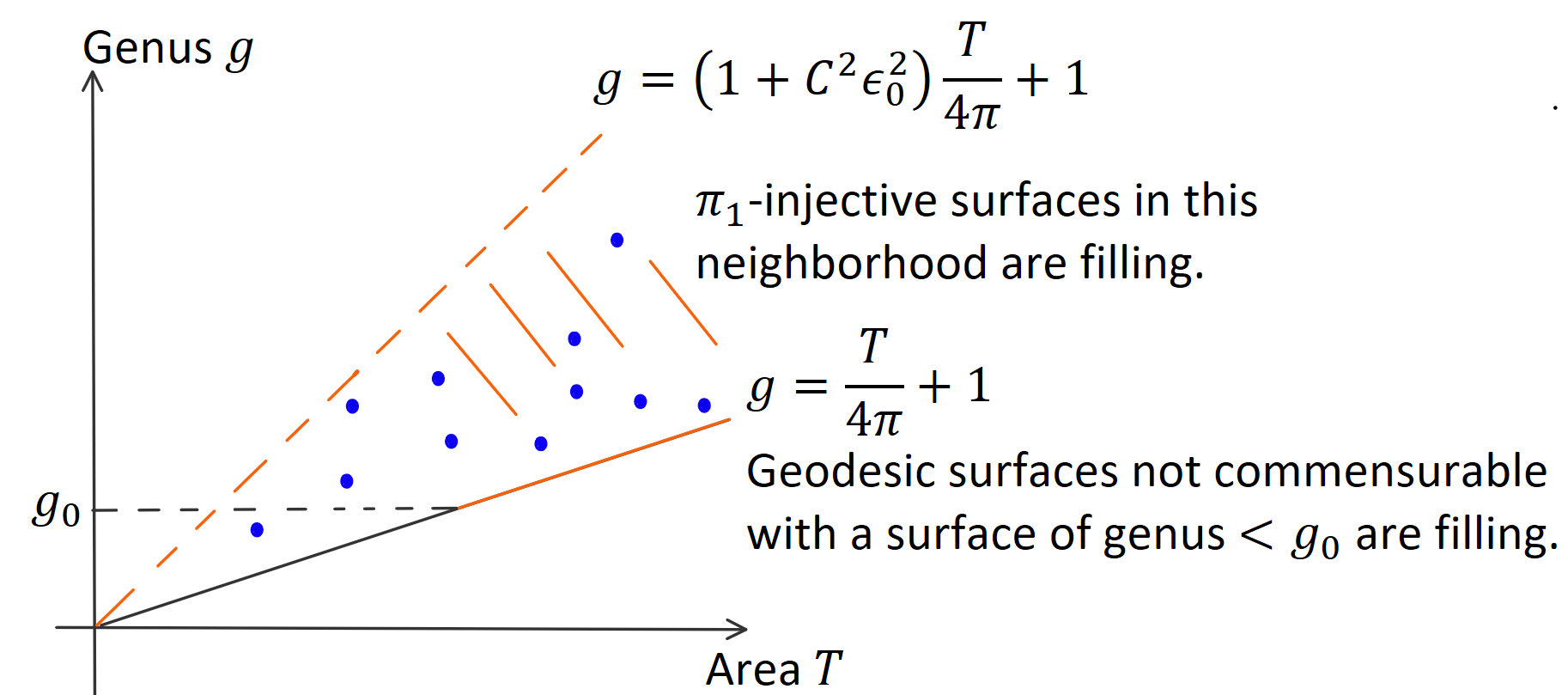}
		\caption{The filling nature of nearly geodesic surfaces in a closed hyperbolic $3$-manifold (each dot represents a $(1+\epsilon_0)$-\qf surface). The constant $C$ is the universal constant coming from Seppi's curvature estimate \cite{saMinimalDiscsHyperbolicSpace}.}
		\label{filling-nbhd}
	\end{figure}
	We summarize \Cref{nearly geodesic are filling} in \Cref{filling-nbhd}. This figure is inspired by [\citenum{kmsGeometricallyTopologicallyRandomsurface}, Figure 1] and justified in \Cref{minimal surface}. 

 Each \cbility classes of Fuchsian surfaces corresponds to a $\psltr$-orbit. A sequence of \eqf minimal surfaces is  \textit{\ag}if $\epsilon\rightarrow 0$, e.g., a sequence of non-\cble Fuchsian surfaces. 
The proof of \Cref{nearly geodesic are filling} uses \Cref{limit of almost geodesic surfaces converge to dense subset} and \Cref{dense implies filling}. 
\begin{theorem}\label{limit of almost geodesic surfaces converge to dense subset}
	Let $M$ be a closed hyperbolic $3$-manifold. If $[S_j]$ is \agn, then the Grassmann $2$-plane bundles $\gt S_j$ over minimal surfaces $S_j$ converge to $\gt M$ in the Hausdorff metric. 
\end{theorem}
\begin{rem}
Independently, Al Assal and Lowe \cite{alAsymptoticallyGeodesicsurfaces} obtain \Cref{limit of almost geodesic surfaces converge to dense subset} for finite-volume hyperbolic manifolds with different perspectives and applications. In contrast, Al Assal \cite{afLimitsAsymptoticalFuchsian} shows that the measure limits of $\gt S_i$ achieve every convex combination of $\sltr$-invariant ergodic measures.
\end{rem}
By a slight abuse of language, by a \emph{plane} we mean an isometric immersion of $\hp$ in $M$ or $\hy$, or its Grassmann $2$-plane bundle in $\gtm$ or $\gt \hy$. In \Cref{limit of almost geodesic surfaces converge to dense subset}, $\gtm$ is the union of uncountably many planes. If, instead, $S_i$ is taken to be the covering of a totally geodesic surface $S$ whose degree tends to infinity as $i\rightarrow \infty$, then the limit of $\gt S_i$ is the single plane $\gt S$. Thus \Cref{limit of almost geodesic surfaces converge to dense subset} reflects the ``one versus uncountably many planes duality" suggested by \mkv to the author. 

\Cref{dense implies filling} then establishes strong fillingness by the following argument. Suppose there exist nontrivial $[\gamma_i]\in \pi_1(M-S_i)$. Then a subsequence of $(x_i,v_i) $ in the unit tangent bundle $T^1\gamma_i$ of the geodesic $\gamma_i$ converges to $(x,v)$ tangent to a geodesic $\gamma$. By \Cref{limit of almost geodesic surfaces converge to dense subset}, there exist $(p_j, \Pi_j) \in \gt S_j \rightarrow (x, \Pi)\perp (x,v)$. This implies that $S_j$ is nearly orthogonal to $\gamma_j$ as well. Then $\partial\tsj$ separates $\partial \tgj$ which means $\pi_1(M-S_j)$ cannot contain $\gamma_j$. We need to use \cite{saMinimalDiscsHyperbolicSpace} which controls the norms of the second fundamental forms of $\tsj$ by $\epsilon_j$ and standard compactness theory for minimal surfaces to prove $\tsj$ graphically  converge to a totally geodesic plane $\tp$. 

 The proof of \Cref{limit of almost geodesic surfaces converge to dense subset} uses crucially \cite{rmRatner'stheorem,snRatnerTheorem} which imply a plane in $M$ is either dense or a closed Fuchsian surface, and uses \cite{msErgodicInvariantMeasure} which implies that a sequence of Fuchsian surfaces in distinct \cbility classes equidistribute (become uniformly dense in $\gt M$). Since $S_i$ are minimal, their Grassmann bundles are connected closed sets, and their Hausdorff limit is a connected closed subset $\gt S$ of $\gtm$. Any finite union of the \grass bundles over Fuchsian surfaces is a disconnected closed set. Thus, $\gt S$, which is closed and $\psltr$-invariant, must contain a dense plane in $\gtm$. 

\Cref{nearly geodesic are filling} is inspired by \Cref{large Fuchsian component}. 
\begin{theorem}\label{large Fuchsian component}
	Let $M$ be a closed hyperbolic $n$-manifold. If $S_i$ is a sequence of non \cble totally geodesic hypersurfaces, then the maximal diameter of components of $M-S_i$ tends to $0$, the number of components tends to infinity, and all but finitely many $[S_i]$ are strongly filling. 
\end{theorem}
Since $M$ is aspherical, a component of $M-S$ of diameter $<2\inj(M)$ is necessarily contractible in $M$. 
 By \cite{bfmsTotallyGeodesicSubmanifolds} (for $n\geq 3$) and \cite{mmArithmeticityHyperbolic} (for $n=3$), if $M$ contains infinitely totally geodesic hypersurfaces, then $M$ is arithmetic of type I. 

An outline of the proof of \Cref{large Fuchsian component} for the dimension $n=3$ is the following. Let $\pi:\hy\rightarrow M$ be the covering projection. Suppose there exist components $C_i \subset M-S_i$ with $\diam(C_i)\geq 2R$. Since $\pi^{-1}(C_i)\subset \hy $ consists of convex components, $C_i$ contains a geodesic segment $\gamma_i$ of length $= 2R$. Up to a subsequence, $\gamma_i $ converge to a geodesic segment $\gamma$ of length $2R$. \Cref{transverse intersection is an open property} proves that the subset of $\gtm$ whose points are tangent to the planes transversally intersecting the central half of $\gamma$ with uniformly bounded angles contains an open set of $\gtm$ of uniformly positive measure (depending only on $R$). Thus eventually $S_i$ must intersect the central half of $\gamma$ with uniformly bounded angles. This shows that $S_i$ must transversally intersect $\gamma_i$ for large $i$, a contradiction. If the maximal diameter of every component of $M-S_i$ is less than $2\inj(M)$, then $S_i$ is filling.

By \Cref{finitely many non filling}, except for finitely many discrete planes (closed totally geodesic surfaces) $P_1, \cdots, P_m$, every plane $P$ in $M$ is strongly filling. \cite{lmPolynomialEffectiveDensity} points out the rate of equidistribution and hence an upper bound on $m$ depends only on the rate of mixing, $\vol(M)$, and the injectivity radius. It is surprising that the techniques of \cite{kmImmersingAlmostGeodesic} and hence the distribution of \km nearly geodesic surfaces also rely on these three geometric invariants. 

\Cref{filling fuchsian surface intersects every geodesic} proves that a filling Fuchsian surface necessarily has some transverse intersections with all geodesics in $M$. Thus all but finitely many planes cut every geodesic of $M$. 

 One way to visualize \Cref{large Fuchsian component} and its transition to \Cref{nearly geodesic are filling} is the following. Let $D$ be a convex polyhedron fundamental domain for $M$ in $\hy$. The generators of $\pi_1 \rightarrow \Gamma \leq \psltc$ are given by finitely many side-pairings of $\partial D$, i.e., $\{ g\in \Gamma: g\bd \cap \bd\subset \partial\bd \tn{ and} \neq \emptyset\} $ (see \Cref{fundamental domain convex with side-pairings}). 
   The preimage in $\hy$ of a filling plane cuts $D$ and $\partial D$ into components so that any geodesic that passes through $D$ must intersect one of the lifts transversally. 

A plane is also a submanifold with principal curvatures $0$. Let $P_1, \cdots, P_m$ (from \Cref{finitely many non filling}) enumerate non filling planes in $M$. The union $K$ of the preimages $\tilde{P_i}$ in $\hy$ of $\cup_{i=1}^m P_i$ is a discrete set of plane. 
If $S$ is a minimal surface with very small  but nonzero principal curvature, its lift $\ts\subset \hy$ cannot be supported entirely in $\cup_{P\in K}N_\epsilon(\tilde{P})$ for $\epsilon$ small enough (see \Cref{ag surfaces not supported in small nbhd of finitely many Fuchsian}). Thus by \Cref{finitely many non filling}, $\ts$ must have a large chunk very close to a plane $P_f$ whose image under $\Gamma$ is strongly filling. Thus the image of $\ts$ under $\Gamma$ is also strongly filling, implying that $S$ is strongly filling. 
 In other words, although the proof of \Cref{large Fuchsian component} relies on the rigidity of $\sltr$-orbits, the topological filling nature persists for small-curvature surfaces which are not homogeneous (i.e., orbits of Lie subgroups of $\psltc$).  

We now describe some motivations from and applications to geometric group theory and virtual properties for strong fillingness. 
We say $[S]$ in $ M$ \textit{separates} two distinct points $p, q \in  \partial \widetilde{M}$ if $S'\in [S]$ has a cover $\widetilde{S'}\subset \tm$ such that  $p, q$ belong to different components of $\partial \widetilde{M}-\partial\widetilde{S'}$. Kahn-\mkv \cite{kmImmersingAlmostGeodesic} construct a collection of nearly geodesic surfaces which is dense in the Grassmannian $2$-plane bundle $\gtm$. 
Sageev \cite{smCodim1Subgroups,smCubeComplex} pioneered in a simple but powerful construction of a cocompact action of a word-hyperbolic group on a CAT(0) cube complex $X$ that is dual to a system of walls associated to a finite collection of codimension-$1$ quasiconvex subgroups. Building on Sageev and motivated by obtaining a proper action on $X$, \bgrn-Wise [\citenum{bwBoundaryCriterionCubulation}, Theorem 1.4] formulate a boundary criterion and point out that these nearly geodesic surfaces satisfy the criterion [\citenum{bwBoundaryCriterionCubulation}, Corollary 4.2]. 

\begin{theorem}[\citenum{kmImmersingAlmostGeodesic,bwBoundaryCriterionCubulation}]\label{separation of geodesics}
	Let $M$ be a closed hyperbolic $3$-manifold. There exist a finite collection $\{[S_1], \cdots, [S_m]\}$ of \qf surfaces  such that every pair of distinct points $p, q \in \partial \widetilde{M}$ is separated by some $\pi_1(S_i)$. 
\end{theorem}
\noindent
Assuming \Cref{separation of geodesics}, Agol \cite{aiVirtualHaken} proves that $\pi_1(M)$ is virtually special (in the sense of Haglund-Wise \cite{hwSpecialCubeComplex}), and hence virtually fibered by \cite{aiCriteriaVirtualFibering}. 
For other important consequences in geometric group theory and virtual properties of $3$-manifolds, see \Cref{cubulations}, \cite{aiVirtualHaken,hwSpecialCubeComplex,smCubeComplex,wdQuasiconvex}, and the references therein. 
There is in general no estimates for the number $m$ in \Cref{separation of geodesics} to separate every pair of points at infinity. If $[S]$ is \qf and strongly filling, then $\{[S]\}$ separates every pair of distinct points at $\partial \tm$ and thus satisfies the Bergeron-Wise boundary criterion.

In \Cref{cubulations} we point out that a \stfs $[S]$ gives a cubulation with a single orbit of hyperplanes (whose stabilizer is essential). Non commensurable \stfs surfaces give rise to  cubulations that are distinct under the isometric actions by $\pi_1(M)$ and thus their existence answers a question of Fioravanti-Hagen on [\citenum{fhDeformingCubulations}, p. 879]. 

\begin{cor}\label{cor cubulations}
	Let $M$ be a closed hyperbolic $3$-manifold or a closed arithmetic hyperbolic manifold with a totally geodesic hypersurface. The space of essential and hyperplane-essential cubulations with a single orbit of hyperplanes whose hyperplane-stabilizers are hypersurface subgroups is infinite up to subdivision and natural isometric actions by $\pi_1(M)$. 
\end{cor}
 A stronger notion of separation than fillingness is \textit{ubiquity}. A collection of immersed surfaces in a hyperbolic $3$–manifold $M $ is \textit{ubiquitous} if, for any pair of disjoint and nonasymptotic planes $\Pi, \Pi' \subset \hy$, some surface in the collection has an embedded preimage $\Tilde{S} \subset \hy$ that separates $\Pi$ from $\Pi'$. Kahn and \mkv \cite{kmImmersingAlmostGeodesic} prove that in a closed hyperbolic $3$-manifold, the set of nearly geodesic surfaces they construct is ubiquitous. Cooper and Futer \cite{cfUbiquitous} prove that in a cusped hyperbolic $3$-manifold, there exists a set of $K$-quasi-Fuchsian surfaces which is ubiquitous for \textit{some} $K>1$, while Kahn and Wright \cite{kwNearlyFuchsian} prove that for any $K>1$,  there exists a ubiquitous set of $K$-quasi-Fuchsian surfaces. 
\begin{cor}\label{ubiquitousSurfaces}
	Let $M$ be a closed hyperbolic $3$-manifold (resp. $n$-manifold). A sequence of \ag  surfaces (totally geodesic hypersurfaces) is ubiquitous. 
\end{cor}
The proofs of \Cref{ubiquitousSurfaces} and \Cref{limit of almost geodesic surfaces converge to dense subset} rely on the rigidity of unipotent flow, i.e., theorems of Ratner \cite{rmRatner'stheorem} and Shah \cite{snRatnerTheorem}, which is very different from the use of geodesic flows in  \cite{kmImmersingAlmostGeodesic,cfUbiquitous,kwNearlyFuchsian} (but \Cref{ubiquitousSurfaces} and \Cref{limit of almost geodesic surfaces converge to dense subset} do not give the existence of nearly geodesic surfaces). The analogous statement to \Cref{ubiquitousSurfaces} for cusped hyperbolic $3$-manifolds is a corollary of the main result of \cite{alAsymptoticallyGeodesicsurfaces}. 

Although the fillingness sorts out a subclass of submanifolds of a hyperbolic $3$-manifold, it also reflects certain topological complexity of the ambient manifold (this principle fails in general). Recall from Thurston [\citenum{twThurstonThreemanifold}, Definition 3.8.1] that a geometric $3$-manifold is a Riemannian $3$-manifold supporting one of the eight model geometries. Nathan M. Dunfield first suggested the existence of \poi filling surfaces (with infinite fundamental groups) in a non hyperbolic geometric $3$-manifold, i.e., the Hantzsche-Wendt manifold. 
\begin{theorem}\label{filling 3-manifold}
	If $M$ is a geometric $3$-manifold and admits a $\pi_1$-injective filling  class of surfaces, then $M$ is hyperbolic, spherical, Euclidean, or $\hp \times \R$.
\end{theorem}
In the proof we explicitly construct some filling surfaces; see \Cref{fillingflat} for the Euclidean manifolds and \Cref{fillinghyperbolic} for the $\hp \times \R$ ones.
We also introduce an elementary inequality connecting the volume of a hyperbolic $3$-manifold with the area of its filling surfaces. By \Cref{filling representative}, a least area surface in a filling $[S]$ cuts $M$ into balls and thus naturally $\area(S)$ and $\vol(M)$ satisfy the isoperimetric inequality for $M$. The hyperbolic space can be characterized as one satisfying linear isoperimetric inequalities. Let $\rank(M)$ be the smallest number of generators of $\pi_1(M)$. 
\begin{theorem}\label{isoperimetric inequality filling surface has area at least volume}
	Let $M$ be a closed hyperbolic $3$-manifold and $[S]$ a filling class of $\pi_1$-injective surfaces. Then any representative $S$ satisfies
	\begin{equation}\label{area of S and M}
		\area(S) > \vol(M),
	\end{equation}
	and there is an absolute constant $C$ such that 
	\begin{equation}\label{rank of a filling surface and hyperbolic manifold}
		\rank(S)>C\cdot\rank(M).
	\end{equation}
\end{theorem}
Current estimates give $C$ an order of $10^{-9}$. The qualitative linearity between the rank of $S$ and $M$ in \eqref{rank of a filling surface and hyperbolic manifold} contrasts other geometries: \Cref{example filling torus in rank tend to infinity} gives a sequence of $3$-manifolds $M_j$ with $\hp\times \R$ geometry such that $\rank(M_j) \rightarrow \infty$ but all admit a filling torus. The proof of \eqref{area of S and M} is a natural consequence of the linear isoperimetric inequality for hyperbolic $3$-space. The proof of \eqref{rank of a filling surface and hyperbolic manifold} relies on \eqref{area of S and M} and \cite{bglsCountingArithmeticLattices} which bound the rank from below by the volume for symmetric space. Cooper \cite{cdPresentationLengthVolume} bounds the volume by the presentation length which inspires \eqref{rank of a filling surface and hyperbolic manifold}. 

 We summarize various theorems for hyperbolic $n$-manifolds in \Cref{table1} for convenience. 
\begin{table}[h]
	\centering
	\begin{tabular}{|m{8cm}|m{3.2cm}|m{3.7cm}|}
		\hline
		\vspace{0.05in}
		{\large Properties} &\vspace{0.05in} {\large $n=2$} &\vspace{0.05in} {\large $n=3$} \\ 	
		\hline
		\vspace{0.05in} Minimizing measures of a hypersurface minimizes its intersection number with closed geodesics.&\vspace{0.05in}  {\Large \checkmark} &\vspace{0.05in} False in general, only true for geodesic submanifolds (\Cref{intersection between geodesic and codimension one submanifold} and \ref{non geodesic least area surface does not minimize intersection number}).	\\
		\hline
		\vspace{0.05in}	Filling geodesic hypersurfaces have some transverse intersections with every  geodesic. &\vspace{0.05in} {\Large\checkmark}  &\vspace{0.05in} {\Large\checkmark} for $n\geq 3$ (\Cref{filling fuchsian surface intersects every geodesic}).\\ 
		\hline
		\vspace{0.05in}	Random geodesic hypersurfaces are filling (Wu and Xue \cite{wxPrimeGeodesicTheorem} asked on a Weil-Petersson random hyperbolic surface of genus $g$, are most closed geodesics of length  $\gg g$ filling). &\vspace{0.05in} 
		Qualitatively, known to experts, e.g. Arana-Herrera  \cite{ahfEffectiveFilling}. Quantitatively, 
		 \cite{dsCountingGeodesics}: almost every closed geodesic of
		length $\gg g(\log g)$ is filling.  &\vspace{0.05in} {\Large\checkmark} in every closed hyperbolic manifold of dim $\geq 3$, for all but finitely many classes (\Cref{large Fuchsian component}). \\ 
		\hline	
		\vspace{0.05in}	There exist infinitely many embedded geodesic hypersurfaces.&\vspace{0.05in} \Large{\checkmark} &\vspace{0.05in} No, not even asymptotically geodesic, \Cref{embedded minimal surfaces of sufficiently small curvature is geodesic}.	 \\
		\hline
		\vspace{0.05in}	There exists a sequence of hypersurfaces whose
		 measure convergence has distinct support from Hausdorff convergence. &\vspace{0.05in} {\Large \checkmark} for many sequences of geodesics by Thurston (\cite{twThurstonThreemanifold,bmIntroductionGeometricTopology}). &\vspace{0.05in} {\Large \checkmark} for asymptotically Fuchsian surfaces (\cite{alAsymptoticallyGeodesicsurfaces} and \Cref{limit of almost geodesic surfaces converge to dense subset}, \cite{alAsymptoticallyGeodesicsurfaces}) unless eventually \cblen. \\
		\hline
	\end{tabular}	
	\caption{Some analogies and contrasts between the hypersurface theory of hyperbolic surfaces and that of hyperbolic $3$-manifolds.}
	\label{table1}
\end{table}

The motivations for \Cref{nearly geodesic are filling} are twofold: one from the statistics of the topological properties of closed geodesics on a random hyperbolic surface, in particular, the work of Wu and Xue \cite{wxPrimeGeodesicTheorem}, and the other from the rigidity of unipotent flows. 
\subsection{The motivations from hyperbolic geometry and random geodesics}
The polar coordinates for the $2$-dimensional hyperbolic metric is 
\begin{equation}\label{polar coordinates}
	g=\di r^2 +\sinh^2 r \di \theta^2.
\end{equation}
A disk of radius $r$ has an area of $2\pi (\cosh r-1)$ and a perimeter of $2\pi \sinh r$. Let $B$ be a compact contractible domain $\subset \hp$ with piecewise smooth boundary $\partial B$. The linear isoperimetric inequality (e.g., [\citen{buser2010geometry}, p. 211]) is
\begin{equation} \label{linear isoperimetric inequality}
	\area(B) <  \ell (\partial B).\end{equation}
 Since a genus-$g$ hyperbolic surface $S_g$ has an area of $4\pi(g-1)$, its fundamental domain $D\subset \hp$ has a boundary length of at least $4\pi(g-1)$. A filling geodesic $\gamma$ cuts the ambient hyperbolic surface into a disjoint union of polygons. Thus by \eqref{linear isoperimetric inequality}, a filling geodesic must have a length of at least $2\pi(g-1)$. On $S_g$, there are only finitely many closed geodesics of length less than $2\pi(g-1)$. Arana-Herrera \cite{ahfEffectiveFilling} suggests that all but a quantitatively small number of primitive closed geodesics of a closed, orientable hyperbolic surface are filling. Quantitatively, Wu and Xue  [\citenum{wxPrimeGeodesicTheorem}, Question on p. 5] ask the following.  
\begin{q}
	As $g \rightarrow \infty$, on a generic closed hyperbolic surface $S$ in the Weil-Petersson model, are
	most closed geodesics of length $\gg g$ filling?
\end{q}
Dozier and Sapir in [\citenum{dsCountingGeodesics}, Section 5.1] construct a sufficiently fine triangulation $\mathcal{D}$ of $S$ and flow boxes centered at the edges of $\mathcal{D}$ so that \textit{most} sufficiently long geodesics pass through all these boxes and hence are filling. Their insightful ideas apply to hypersurfaces in higher-dimensional manifolds and detect fillingness. However, for \Cref{large Fuchsian component,nearly geodesic are filling}, we have adapted an argument different from \cite{dsCountingGeodesics} to prove large totally geodesic surfaces are filling by taking advantage of the rigidity of $\sltr$-orbits explained below. The crucial difference is that, given any prescribed subset $O$ of $\gtm$, all large \toge  surfaces \textit{must} get close and nearly tangent to $O$. 

\subsection{The motivations from the rigidity of unipotent flows}\label{motivationRatner}
There is a huge variety of invariant sets of geodesic flows on the unit tangent bundle $T^1S$ of a hyperbolic surface, i.e., dense [\citenum{khModernDynamics}, 17.5, Corollary 18.3.5], closed geodesics, geodesic laminations [\citenum{bmIntroductionGeometricTopology}, 8.3], and sets with Hausdorff dimension greater than $1$ but less than $3$ [\citenum{lsVariationsTheorem}, Theorem 1.3]. The flexibility of such invariant sets is also reflected by the existence of arbitrarily long closed geodesics supported on a proper subsurface of $S$, which are never filling. Although Fuchsian surfaces have principal curvature $0$ and geometrically generalize geodesics, it is their connection with dynamic systems that gives rise to their filling topological nature in \Cref{large Fuchsian component} and \Cref{nearly geodesic are filling}. 

Let $G$ be a connected Lie group and $U$ a unipotent one-parameter subgroup of $G$. In contrast with geodesic flows, there is a strong rigidity for the orbits of horocycle flow ($U$-orbits): Hedlund's theorem \cite{hgHorocyclesFlowMinimal} proves the horocycle flow on a closed hyperbolic surface $S$ is minimal, i.e., every horocycle on $S$ is dense. Let $SO_0(n,1)$ denote the connected component of the group of linear transformations on $\R^{n+1}$ preserving the bilinear forms $\langle x, y \rangle = \sum_{i=1}^{n}x_iy_i-x_{n+1}y_{n+1}.$ 
Shah \cite{snRatnerTheorem} classifies the orbit closure of $SO_0(n-1,1)$ for actions on finite-volume  quotients of $SO_0(n,1)$ (the framed bundle of $\mathbb{H}^n$) for $n\geq 3$; Etienne Ghys pointed out the implications in geometry: the projection of a codimension-one geodesic hyperplane $P\subset \mathbb{H}^n$ to a closed hyperbolic $n$-manifold $M$ is either dense or a closed totally geodesic hypersurface. For $n \geq 3$, the stark contrast between $SO_0(n-1, 1)$ and $SO_0(1,1)$ (which acts as the geodesic flow) relies crucially on the behavior of the orbits of nontrivial unipotent $1$-parameter subgroups $U$ contained in $SO_0(n-1,1)$, as can be seen in [\citenum{snRatnerTheorem}, Section 4-6] (the linearization of $U$ is a polynomial, allowing non-divergence estimates which do not hold for geodesic flows).   
Margulis \cite{mgDiscreteSubgroupsErgodicTheory} shows how to effectively study orbit closures in homogeneous spaces of Lie groups by analyzing minimal closed invariant sets of the action of unipotent subgroups. Some other important works include, for example, Dani-Smillie, Dani, Dani-Margulis, which we refer the readers to the references in e.g., \cite{mgDiscreteSubgroupsErgodicTheory,snRatnerTheorem,snUniformlydistributedOrbits,rmRatner'stheorem}. 

 Ratner's celebrated breakthrough \cite{rmRatner'stheorem} classifies finite ergodic $U$-invariant measures and proves every unipotent subgroup $U$ of a connected Lie group $G$ is strictly measure rigid, i.e., for \emph{every} discrete subgroup $\Gamma\leq G$ and every $U$-invariant Borel probability measure $\mu$ on $\Gamma\backslash G$, the support of $\mu$ is a closed orbit of the subgroup of $G$ preserving $\mu$.   Building on \cite{rmRatner'stheorem}, \cite{msErgodicInvariantMeasure} implies that in a finite-volume hyperbolic $n$-manifold $M$, a sequence of non-\cble totally geodesic hypersurfaces $S_i$ becomes equidistributed, i.e., the probability measure induced by the inclusion of $\mathcal{G}_{n-1}S_i$ on the Grassmann bundle $\mathcal{G}_{n-1}M$ converges to the Liouville measure on $\mathcal{G}_{n-1}M$. 
 
 Calegari-Marques-Neves \cite{cmnCountingminimal} initiated the studies of the distribution of minimal surfaces in negatively curved manifolds using homogeneous dynamics and rigidity results. 
 Lowe \cite{lbDeformationsTotallyGeodesic} studies the minimal foliations of Grassmann bundles of negatively curved manifolds. He also proves various quantitative density and measure convergence results for \ag surfaces in non-homogeneous metrics, e.g., [\citenum{lbDeformationsTotallyGeodesic}, Theorem 1.9]. Jiang \cite{jrMinimalSurfaceEntropyCusped} generalizes the minimal surface entropy results of \cite{cmnCountingminimal} to higher-dimensional and  noncompact finite-volume hyperbolic manifolds. Thurston in [\citenum{twThurstonThreemanifold}, Remark after Corollary 8.8.6] points out that unlike closed hyperbolic $3$-manifolds, a noncompact finite-volume hyperbolic manifold might have infinitely many homotopy classes of $\pi_1$-injective surfaces of bounded genus.  Jiang \cite{jrMinimalSurfaceEntropyCusped} proves that the number of homotopy classes of nearly geodesic surfaces of bounded genus remains finite. Lowe  \cite{lbDeformationsTotallyGeodesic}, following ideas of Labourie \cite{lfAsymptoticCounting} and Lowe-Neves [\citenum{lnMinSurEntropy}], proves that the weak-$^*$ limit of the probability measures induced by a sequence of \ag surfaces is $\psltr$-invariant. Al Assal \cite{alAsymptoticallyGeodesicsurfaces} proves that the space of such limiting measures is exactly the space of $\psltr$-invariant measures.

An outline of the paper is the following. In \Cref{prelim}, we recall some background on \qf and minimal surfaces, and the metric and measure on Grassmann bundles. In \Cref{proofs}, we prove \Cref{large Fuchsian component,nearly geodesic are filling}, \Cref{embedded minimal surfaces of sufficiently small curvature is geodesic}, \ref{gap min surface}, \ref{cor cubulations}, and \ref{ubiquitousSurfaces}. In \Cref{other}, we classify geometric $3$-manifolds containing a filling $\pi_1$-injective surface, construct explicit filling surfaces in Euclidean and $\hp\times \R $ geometry, and prove \Cref{filling 3-manifold,isoperimetric inequality filling surface has area at least volume}.
\section*{Acknowledgment}
The author would like to thank Wenyuan Yang for suggesting a connection between \Cref{large Fuchsian component} and the work of Kahn-\mkvn. He thanks \vl \mkv for insightful discussions on the asymptotic behaviors of the planes which help formulate \Cref{limit of almost geodesic surfaces converge to dense subset} and part of \Cref{nearly geodesic are filling}. He also thanks Nathan M. Dunfield, Mark Hagen, Zhenghao Rao, and Yibo Zhang for many helpful discussions. He thanks his postdoc mentor Yunhui Wu for many helpful discussions in the last few years on hyperbolic geometry and minimal surfaces. The author is truly grateful for anonymous referee's very long report which significantly improves the paper in many aspects. 
Part of the work is completed when the author visits the Institut Fourier. The author would like to thank the institute for its hospitality and thank Andrea Seppi and Zeno Huang for their interest in this work and for helpful discussions. The author is partially supported by the start-up grants at Shanghai Institute of Mathematics and Interdisciplinary Sciences and  NSFC No. 12501084. 

 \tableofcontents
    \section{Preliminaries} \label{prelim}
\subsection{Quasi-Fuchsian surfaces and minimal surfaces}\label{minimal surface}
A surface $S \hookrightarrow M$ is \textit{$\pi_1$-injective} if the induced homomorphism $\pi_1(S)\rightarrow \pi_1(M)$ is injective.
 We call $S\hookrightarrow M$ \textit{primitive} if $\pi_1(S)$ is not a proper subgroup of the fundamental group of another orientable subsurface.
Since $M$ is hyperbolic, a Fuchsian surface $S$ of genus $g$ has intrinsic sectional curvature $-1$. The conjugacy classes of surface subgroups $\pi_1(S)$ of $\pi_1(M)$ are in one-to-one correspondence with the homotopy classes of \poi immersions $f: S\rightarrow M$. We define a local variation of $f(S)$ to be a smooth map 
$$F: S \times (-\eta, \eta)\rightarrow M \quad (\eta >0)$$ with $F(x,0)=x$ for all  $x\in S.$
For sufficiently small $|t|$, $\Psi_t(\cdot)\coloneqq F(\cdot, t)$ defines a smooth immersion of $S$ into $M$. A surface $S$ is \textit{minimal} if 
$$\frac{\di}{\di t}\area(\Psi_t(S))|_{t=0} =0 $$
for every local variation $\Psi$. A $\pi_1$-injective surface $S$ is \textit{least area} if $S$ achieves the smallest area in its homotopy class (which is then necessarily minimal).
 By  \cite{suMinimalImmersionClosedRiemann,syExistIncomMinTopo}, a $\pi_1$-injective immersion $f: S\rightarrow M$ is homotopic to a smoothly immersed least area surface. Henceforth, we assume all such immersions $f: S \rightarrow M$ are minimal unless stated otherwise.  A surface is \textit{Fuchsian} or \textit{totally geodesic} if any geodesic on $S$ with respect to its induced Riemannian metric is also a geodesic on $M$. For further background on minimal surfaces, see \cite{cmCourseMinSurface} and [\citenum{jjRiemannianGeoAnalysis}, 5.4].

An immersed surface $f: S\hookrightarrow M$ is \textit{\eqf}if the limit set of a universal cover $\tilde{S} \subset \hy $ is a $(1+\epsilon)$-quasicircle. Two $\pi_1$-injective surfaces $S_1, S_2 \subset M$ are \textit{commensurable} if there exists a $g\in \pi_1(M)$ such that $g\pi_1(S_1)g^{-1}\cap\pi_1(S_2)$ is of finite index in both $\pi_1(S_1)$ and $\pi_1(S_2)$. Commensurable surfaces share the same \qf constants. 

Seppi proves the following theorem.
\begin{theorem}[\citenum{saMinimalDiscsHyperbolicSpace}, Theorem A]\label{seppi curvature}\label{seppi curvature estimates}
	There exist universal positive constants $\epsilon_1$ and $C$ such that for every $(1+\epsilon)$-\qf minimal surface with $\epsilon\leq \epsilon_1$, its supremum principal curvature $\lambda_0\coloneqq	\|\lambda\|_{L^{\infty}}$ satisfies
	\begin{equation}\label{eq seppi curvature}
	 \lambda_0 \leq C\epsilon.
	\end{equation}
\end{theorem}
 
Conversely, Z. Huang and B. Wang prove 
\begin{theorem}[\citenum{hwAlmostFuchsianManifolds}, (3.16)]\label{huangwangCurvatureupperboundQF}
	If a $(1+\epsilon)$-\qf minimal surface $S$ satisfies $0<\lambda_0<1$, then
	$$ 1+\epsilon< \frac{1+\lambda_0}{1-\lambda_0}.$$ 
\end{theorem}
\noindent
Both $\epsilon=0$ and $\lambda_0=0$ correspond to totally geodesic surfaces. In a hyperbolic manifold, by the Gauss-Bonnet theorem and the Gauss equation, the area $T$ of a minimal surface $S$ is $\leq 4\pi(g-1)$, with equality if and only if $S$ is totally geodesic. Moreover, combining with \Cref{seppi curvature estimates}, we can derive that the area of a $(1+\epsilon_0)$-\qf surface is $\geq (\frac{1}{1+C^2\epsilon_0^2})4\pi(g-1)$. 
 This justifies \Cref{filling-nbhd}. 

A minimal surface with $\lambda_0 <1$ is \textit{almost-Fuchsian}. Uhlenbeck \cite{ukclosedminSurfaceHyperbolic} pioneered the studies of such surfaces, using the metric and the second fundamental form of an almost-Fuchsian surface in a quasi-Fuchsian $3$-manifold $M$ to parameterize the metric tensor of $M$. An unpublished result of Thurston states that almost-Fuchsian surfaces are $\pi_1$-injective, see [\citenum{lcSmallCurvatureSurface}, Theorem 5.1]. 
\begin{theorem}[\cite{ukclosedminSurfaceHyperbolic} and \cite{saMinimalDiscsHyperbolicSpace}]\label{uniqueness}
	Let $M$ be a closed hyperbolic $3$-manifold. There exists an   $\epsilon_2>0$ (depending only on $M$) such that every $(1+\epsilon
	_2)-\qf$ surface is homotopic to a unique minimal surface. 
\end{theorem}
This theorem allows us to speak of the \emph{distribution} of \ag surfaces. 
Throughout this work, we assume $\epsilon_0$ (e.g., in \Cref{nearly geodesic are filling}) satisfies $\epsilon_0\leq \min\{\epsilon_1,\epsilon_2\}$. 
 For recent advances on almost-Fuchsian manifolds and surfaces, see \cite{hlsUniqueness} and the references therein. 

When analyzing filling properties and geodesic intersections, there is no loss of generality in considering orientable surfaces of minimal genus within each class modulo homotopy and commensurability. On a hyperbolic surface $S$, a primitive closed geodesic $\gamma$ and its cover $\gamma^n$ span identical subsets of $S$. Similarly, it follows from the uniqueness of almost Fuchsian minimal surface in its homotopy class that a \las primitive almost Fuchsian surface $S$ and commensurable \las surfaces span identical subsets of $M$, and their preimages in $\hy$ are the same by [\citenum{fhsLeastAreaImmersion}, Theorem 5.3]. 
  A hypersurface of a Riemannian manifold is a closed codimension-$1$ submanifold with the induced Riemannian metric.
\subsection{A natural metric on the Grassmann bundles}\label{grassmann bundle}
Following \cite{pgGeodesicFlows}, we construct a natural metric on the Grassmann $2$-plane bundle for the convenience of the readers. 
 A point in $\gtm$ is denoted by $(x, \Pi)$ where $x\in M$ and $\Pi\subset T_xM$ is a $2$-plane. The unit tangent bundle $T^1M$ is a $2$-fold covering of $\gtm$. 

For this subsection only, let $\pi: \gtm \rightarrow M$ be the canonical projection, i.e., if $\theta =(x, \Pi_1) \in \gtm$, then $\pi(\theta)=x$. There exists a canonical subbundle of $T\gtm$ called the \textit{vertical subbundle}: its fiber at $\theta$ is given by the tangent vectors of curves $\sigma:(-\eta, \eta) \rightarrow \gtm$ of the form $\sigma(t) = (x, \Pi_1+t\Pi_2)$, where $\Pi_2 \in (\gt)_xM$. Equivalently, 
$$V(\theta)= \ker(d_{\theta}\pi).$$ 
 We now use the Riemannian metric on $M$ to construct a metric on $\gtm$ and a complementary horizontal subbundle. There is a horizontal lift 
$$L_{\theta}: (\gt)_xM \rightarrow T_{\theta}\gtm$$
using parallel transport as follows. Given $\Pi \in \gtxm$ and $\alpha: (-\eta, \eta)\rightarrow M$ such that $\alpha'(0) \perp \Pi$ in $T_xM$, let $Z(t)$ be the parallel transport of $\Pi$ along $\alpha$. Let $\sigma: (-\eta, \eta)\rightarrow \gtm$ be the curve 
$$\sigma(t) = (\alpha(t), Z(t)).$$ 
Then, with $\theta=(x, \Pi)$, 
$$ L_{\theta}(\Pi)\coloneqq \sigma'(0) \in T_{\theta}\gtm. $$ 
The horizontal subbundle is defined as $\im(L_{\theta})$. By [\citenum{pgGeodesicFlows}, Lemma 1.15], we see that as vector spaces 
$$T_{\theta}\gtm = H(\theta) \oplus V(\theta). $$
Equip each fiber with $\frac{1}{2\pi}$ times the standard Riemannian metric on $\R P^2$, so that the volume measure of each fiber $(\gt)_x$ is $1$.
\begin{defi}[A natural invariant metric on the Grassmann bundle $\gtm$]\label{grassmann metric}
	Using the decomposition $T_{\theta}\gtm = H(\theta) \oplus V(\theta)$, we define the Sasaki metric on $\gtm$ by declaring $H(\theta) \oplus V(\theta)$ to be orthogonal. Denote this metric by $\langle \cdot \, ,\cdot \rangle_{\gtm}$. 
\end{defi}
The orthogonal splitting endows $\gtm$ with a natural metric $\langle \cdot \, ,\cdot \rangle_{\gtm}$ satisfying the following properties:
\begin{enumerate}
	\item the map $\pi: \gtm \rightarrow M$ is a Riemannian submersion with totally geodesic fibers,
	\item the projection of a geodesic in $\gtm$ to $M$ is a geodesic, in particular, 
	$$ \dist_{\hy}(x_1, x_2) \leq \dist_{\gt\hy}((x_1, \Pi_1), (x_2, \Pi_2)),$$ and
	\item the lift of this metric to $\gt(\hy)$ is invariant under $\isom(\hy)$, and agrees with hyperbolic metric on totally geodesic planes in $\gt(\hy)$.
\end{enumerate}
 The Sasaki metric induces a probability measure on $\gtm$ called the \textit{Liouville measure}, invariant under the geodesic flow. 

Let $f: S\rightarrow M$ be an \poi quasi-Fuchsian surface. There is a hyperbolic metric on $S$ with respect to which $f$ is a conformal minimal immersion. Here $f$ also induces a natural map to $\gtm$, which we still denote by $f$. 

\begin{defi}[Probability measure induced by minimal surfaces]
	For a measurable subset $O\subset \gtm$, and $f: S\rightarrow \gtm$, we have
	$$
	\ds\mu_{S}(O)= \frac{\area(f^{-1}(O\cap \gt S))}{\area(S)}.
	$$
\end{defi}
There is an identification between $\psltc$ and the frame bundle $\fhy$ over $\hy$. Moreover, $\psltc$ acts on $\fhy$ by isometries freely and transitively on the left. There is a right action on $\fhy$ by various flows, e.g., geodesic flow and unipotent flows. The bundle $\gt \hy$ is a natural quotient of $\fhy$: $\gt\hy=\fhy/SO(2)$. Each element in $\gt \hy$ is tangent to a unique geodesic plane $P$, an isometric embedding of $\hp$ in $\hy$. Since each $\hp$ is a quotient of $\psltr$, by a $\psltr$-orbit in $\gt\hy$, we mean the projection of an orbit of $\psltr$ acting on $\fhy$ to $\gt\hy$. A lattice $\Gamma \leq \psltc$ acts on $\fhy$ or $\gt \hy$ on the left, and the orbits descend to $\gtm\coloneqq \Gamma \backslash \fhy /SO(2)$. 

\subsection{Fundamental domains and fillingness}
We recall some results about fundamental domains of a hyperbolic manifold $M$, which help us visualize how a filling surface interacts with geodesics and $M$. 

By [\citenum{rjFoundationsHyperbolicManifolds}, Theorem 6.6.13], a discrete subgroup of $\isom(\hs)$ admits a locally finite Dirichlet fundamental domain $D$. By [\citenum{rjFoundationsHyperbolicManifolds}, Exercise 6.6.2], the discrete group $\Gamma=\pi_1(M)$ is generated by side-pairings: 
$$\{g\in \Gamma : \bar{D} \cap g\bar{D} \subset \partial\bd \tn{ and}\neq \emptyset\}.$$
By the first paragraph of [\citenum{rjFoundationsHyperbolicManifolds}, 6.7], there exists a convex Dirichlet fundamental domain $D$. Furthermore,  [\citenum{rjFoundationsHyperbolicManifolds}, Theorem 6.7.1] proves that $\bar{D}$ is a convex polyhedron. Finally,  [\citenum{rjFoundationsHyperbolicManifolds}, Theorem 6.7.4(2), Theorem 6.7.5] prove that the fundamental domain $D$ is exact: for any side $F$ of $D$, there exists a unique $g\in \Gamma$ such that 
$$ F=\bar{D} \cap g\bar{D}.$$
This culminates in the following lemma. 
\begin{lem}\label{fundamental domain convex with side-pairings}
	A closed hyperbolic $n$-manifold admits a fundamental domain $D \subset \hs$ which is a convex polyhedron, whose sides are partitioned into pairs $\{ \{F_1, F_1'\}, \cdots, \{F_k, F_k'\}\}$, such that for each pair of sides $F_i, F_i'$, there exists a unique $g_i \in \Gamma$ which maps $F_i$ isometrically onto $F_i'$. Moreover, the group $\Gamma$ is generated by  $\{g_1, \cdots, g_k\}$. 
\end{lem}
A surface representative $S\hookrightarrow M$ is \textit{filling} if $M-S$ consists of components that are contractible in $M$.
A surface $S\hookrightarrow M$ is \textit{genuinely filling} if $M-S$ consists of topological balls. We always assume surfaces in hyperbolic $3$-manifolds are quasi-Fuchsian. Let $\pi: \hy \rightarrow M$ be the covering projection. Then $\pi(\partial D)$ is a \gefi closed surface which is homotopically trivial. This motivates us to define a filling class of surfaces as a property that should hold for every representative. We thank Chao Wang for clarifications regarding  filling surfaces versus filling homotopy classes. A filling class $[S]$ intersects every homotopically nontrivial loop of $M$ which means $\pi_1(M-S)\rightarrow \pi_1(M)$ is trivial. We record the following useful lemma from Rubinstein-Sageev. 
\begin{lem}[\citenum{rsFillingSurfaces}]\label{filling representative}
	Let $M$ be a closed \htms and $[S]$ a filling class of surfaces. Then a \las representative $S$ is \gefin. 
\end{lem}
Thus, similar to the boundary of a fundamental domain, any surface in a filling class has an area lower bound by the linear isoperimetric inequality (see \Cref{isoperimetric inequality filling surface has area at least volume}). 
	\section{Strongly filling surfaces in hyperbolic $3$-manifolds}\label{proofs}
We start with some lemmas on Fuchsian and filling surfaces, whose proofs also help us visualize the proof of \Cref{nearly geodesic are filling}. 
\subsection{Some lemmas on intersections and fillingness}
In this subsection, with a slight abuse of language, we also use $[S]$ to denote a homotopy class of surfaces (instead of modulo both homotopy and commensurability). We first define the intersection number between $[\gamma] \in \pi_1(M)$ and a homotopy class $[S]$ of immersed surfaces. By applying homotopies, we may assume an immersed closed curve $\gamma:S^1 \rightarrow M$ and an immersed surface $f:S\rightarrow M$ intersect at finitely many points $x_1, \cdots, x_n$, such that for each $i\in\{1, \cdots, n\}$, there exist $j\in\{1, \cdots, I_i(\gamma)\}$, $k\in\{1, \cdots, I_i(S)\}, t_{ij}\in S^1$, and $y_{ik} \in S$ with $$\gamma(t_{ij})=x_i = f(y_{ik}) \tn{ and } \gamma'(t_{ij}) \oplus df_{y_{ik}}(T_{y_{ik}}S)=T_{x_i}M.$$ 
The \textit{intersection number} of $\gamma$ and $S$ is defined by counting with multiplicity: 
$$
	|\gamma \cap S| = \sum_{i=1}^n I_i(\gamma) \cdot I_i(S).
$$

\begin{defi}\label{geometric intersection number}
 The intersection number between a homotopy class $[\gamma] \in \pi_1(M)$ and a homotopy class  $[S]$ of surfaces is $$[\gamma]\cdot	[S]\coloneqq\min\{|\gamma' \cap S'|:\gamma' \in [\gamma], S'\in [S] \}.$$ 
	\end{defi}
 A Fuchsian surface minimizes the intersection number with a closed geodesic in its homotopy class.
\begin{lem}\label{intersection between geodesic and codimension one submanifold}
	Let $[\gamma]\in \pi_1(M)$ and $[S]$ be a homotopy class of a Fuchsian surface. Then the intersection number $[\gamma]\cdot [S]$ is realized by the number of transverse intersections of their geodesic representatives $\gamma$ and $S$
	$$
	[\gamma]\cdot	[S] = |\gamma\cap S|.
	$$ 
\end{lem}
\begin{proof}
	Lift $M$ to a fundamental domain $D \subset \hy$. The preimages of $\gamma$ and $S$ in $D$ are finitely many geodesic arcs and geodesic polygons. If an arc which is a subset of a geodesic line $\tg $ has a transverse intersection with a plane $P$, then the circle $\partial P$ separates the two endpoints of $\tg$. By the Morse lemma, any homotopy of $\gamma$ induces a homotopy of $\tg$, fixing its limit points. Similarly, any homotopy of $S$ induces a homotopy of $P$, fixing the circle at infinity. Thus no homotopy decreases the intersection number between $\beta $ and $P$. Applying this argument to all transverse intersections between the preimages of $\gamma$ and the preimages of $S$ through $D$, we conclude that the minimal number is realized by taking the geodesic representatives of both $\gamma$ and $S$. 
\end{proof}
We call the intersection between a Fuchsian surface and a closed geodesic in a hyperbolic manifold \textit{essential}. The strong interplay between topology and geometry is a recurring theme in hyperbolic geometry. If $M=S^2 \times S^1$ and $\gamma\subset S^2$ a geodesic great circle, then $S\coloneqq\gamma\times S^1 \subset M$ is a compressible totally geodesic torus. Thus for another geodesic $\eta$, $[\eta]\cdot [S]<|\eta \cap S|$ and  \Cref{intersection between geodesic and codimension one submanifold} fails in general. 

As a corollary of \Cref{intersection between geodesic and codimension one submanifold}, we have 
\begin{lem}\label{Fuchsian transverse intersect means complement has no that loop}
If a Fuchsian surface $S$ intersects a geodesic $\gamma$ transversally, then for any $S'$ homotopic to $S$, $\pi_1(M-S')$ does not contain $[\gamma]$. Moreover, if a Fuchsian surface $S$ is \gefin, its homotopy class and \cbility class are strongly filling. 
\end{lem}
\noindent
Thus if a Fuchsian $S$ transversally intersects every closed geodesic of $M$, $[S]$ is filling. This local geometric criterion for a homotopy class of geodesic surfaces to be filling also generalizes to certain non hyperbolic manifolds, e.g., the proof of \Cref{three families to be filling}. 
\begin{rem}
\Cref{intersection between geodesic and codimension one submanifold} and \ref{Fuchsian transverse intersect means complement has no that loop} do not generalize to \las non-Fuchsian surfaces in hyperbolic manifolds: \Cref{non geodesic least area surface does not minimize intersection number} proves every such surface has redundant intersections with some closed geodesics. 
\end{rem}

\begin{lem}\label{covering preserves filling surfaces}
	Let $S\hookrightarrow M$ be a filling surface and $M'$ a cover of $M$. Then the preimage of $S$ in $M'$ is a (possibly disconnected or noncompact) filling surface $S' \subset M'$. 
\end{lem}
\begin{proof}
	Denote by $\pi: M' \rightarrow M$ the covering projection. If $S' = \pi^{-1}(S) $ is not filling, then there exists a non-contractible component $B$ of $M'-S'$. Then $\pi(B)\subset M-S$ is non-contractible, which is a contradiction.  
\end{proof}
 Let $\grm$ be the Grassmann $2$-plane bundle over $M$. We denote by
 \begin{enumerate}
 	\item $\pi: \hy \rightarrow M$ the universal covering map, and 
 	\item $\pr: \grm \rightarrow M$ the natural projection.
 \end{enumerate}
  Define the \textbf{set of planes} in $\grm$ by 
\begin{equation}
	\planes(M)\coloneqq \grm /\sim
\end{equation}
where $(x_1, \Pi_1) \sim (x_2, \Pi_2)$ if they are belong to the same $\psltr$-orbit in $\gtm$. This is a three-dimensional non Hausdorff space. By a slight abuse of language, we identify a plane with the subset it is tangent to in $\gtm$. If $\Pi_x$ is a $2$-plane tangent at $x$, then $P(\Pi_x)$ is the plane in $\gtm$ passing through $\Pi_x$. By identifying the frame bundle of $\hy$ as $\psltc$, we identify $\planes(\hy) \cong \psltc/\psltr$ and $\planes(M)=\Gamma\backslash \psltc /\psltr  $. The projection map $\grm \rightarrow \planes(M)$ is continuous, and the fiber of the projection map is isomorphic to $\hp$. 

For convenience later, if a plane $P$ corresponds to a closed Fuchsian surface $S$, we call $P$ \emph{discrete}. The genus of a discrete plane is then defined as the genus $\gs$. 
 Since there are only finitely many Fuchsian surfaces of genus up to $g$, and no distinct \cbility classes of Fuchsian surfaces are tangent, the set of planes up to genus $g$ is a \textbf{disconnected closed set} in $\grm$ (when this set has more than one planes). A subset $B$ of $\grm$ is  filling if $\pr(B)\subset M$ is filling. For $x\in M$ and $\Pi_x \in \grm$ tangent to a plane $P$, we say $\Pi_x$ intersects a geodesic $\gamma$ transversally in $M$ if $\gamma$ and $P$ intersect at some point transversally. Thus for a filling Fuchsian surface $S$, any $x\in S$, and any geodesic $\gamma \subset S$, $\Pi_x$ intersects $\gamma$ transversally by \Cref{filling fuchsian surface intersects every geodesic}. 
\subsection{A filling \toge surface is \stf}
Recall that a homotopy class $[S]$ of surfaces in a hyperbolic manifold $M$ \textit{separates} $p, q \in  \partial \widetilde{M}$ if any $S'\in [S]$ has a cover $\widetilde{S}$ such that $\partial \widetilde{S}$ separate $p, q \in \partial \widetilde{M}$. Contractible surfaces in $M$ can be homotoped to be filling.  However, [\citenum{rsFillingSurfaces}, Lemma 2.1] proved that if a Fuchsian surface $S$ is filling, its class $[S]$ is \stf. We present a slightly simpler proof. 
\begin{lem}\label{filling fuchsian surface intersects every geodesic}
	Let $M$ be a closed hyperbolic $n$-manifold. If a totally geodesic hypersurface $S\subset M$ is filling, it has some transverse intersections with every geodesic in $M$. In particular, a filling totally geodesic hypersurface is \stf.
\end{lem}
\begin{proof}
	Let $S$ be a filling totally geodesic hypersurface. By a plane of $\mathbb{H}^n$, we mean a totally geodesic hyperplane. Since $S$ is immersed, each pair of planes in $\mathbb{H}^n$ which cover $S$ has either empty or transverse intersections. Suppose 
	$$M-S =\cup_{i=1}^n B_i , \tn{ where } B_i \tn{ is open and contractible in }M, $$ 
	and $\partial B_i= \cup_{j=1}^m F_j$ where $F_j$ are geodesic polyhedrons.
	Since $B_i$ is contractible in $M$, $B_i$ lifts isometrically to $\hs$, which implies that $B_i$ is convex, and hence contractible as a  topological space itself. Thus $S$ is \gefin. 
	
	Let $\tg$ be a geodesic in $\mathbb{H}^n$ and $\gamma = \pi(\tg)$. The geodesic $\gamma$ cannot entirely lie in any $B_i$, since otherwise, by lifting $B_i$ and $\gamma$ to $\mathbb{H}^n$, $\tg$ is a subset of a bounded set in $\mathbb{H}^n$, a contradiction. 
	
	Thus $\gamma\cap \partial B_i \neq \emptyset$ for some $i$.  If $\gamma$ transversally intersects the interior of some polyhedron $F_j$ transversally, by lifting $\gamma$ and $F_j$ to $\hs$, we find a plane tangent to $\widetilde{F_j}$ which separates $\partial\tg$. 
	
	Each $F_j$ is compact and contractible. If $\gamma \subset F_j$, since $\gamma$ is complete, $\gamma$ must intersect an $(n-2)$-dimensional polyhedron $\alpha \subset \partial F_j$. However, $\partial F_j$ is a subset of the self-intersection locus of $S$, which means there is another face $F_k$ so that $\alpha \subset F_j \cap F_k$. Thus $\gamma$ intersects $F_k$ transversally. By lifting, we find a plane separating $\partial\tg$.  
\end{proof}
\begin{cor}
Let $M$ be a closed hyperbolic $n$-manifold. A totally geodesic hypersurface $S\hookrightarrow M$ is filling if and only if every component of $\tm - \pi^{-1}(S)$ is bounded. 
\end{cor}

\begin{rem}
	The argument above relies on the analysis of the intersection of a geodesic $\gamma$ and a geodesic face $\subset$ a plane $P$ and the dichotomy that either $\tg$ is a subset of $ \tp$ or $\tg$ transversally intersects $\tp$. This gives a geometric criterion for a homotopy class of totally geodesic surfaces to be filling in certain aspherical manifolds, see \Cref{non hyperbolic filling} for Euclidean and $\hp\times \R$ $3$-manifolds.  
\end{rem}

\subsection{All but finitely many totally geodesic hypersurfaces are strongly filling}\label{proof Large Fuchsian filling}
One way to see that as the genus of discrete planes grows, they must cut $M$ into more and more components is as follows. Using \Cref{fundamental domain convex with side-pairings}, $M$ admits a fundamental domain $D\subset \hy$, which is a compact convex polyhedron of diameter $d$. The intersection of $D$ with any plane in $\hy$ has an area upper bound $2\pi (\cosh \frac{d}{2}-1)$ using the polar coordinates \eqref{polar coordinates}. As the area of $S$ grows, the preimage of $S$ in $D$ consists of more and more geodesic polygons in $D$. \cite{msErgodicInvariantMeasure} implies that these polygons are uniformly distributed and thus must cut $M$ into components of decreasing diameter. 

For $R>0$, define two positive functions $$h_1(R)=\cot^{-1}\left(\frac{\sinh \frac{R}{4} \cosh \frac{R}{2}}{\sinh \frac{R}{2}} \right) -\cot^{-1}\left(2\cosh^2 \frac{R}{4} \right) $$ and $$h_2(R)= \frac{\pi(1-\cos\theta_0)}{4}(R\cosh R-R),$$
where $\theta_0= \cot^{-1}\left(2\cosh^2 \frac{R}{4}\right)$. 
The following lemma works similarly in every dimension. For the clarity of presentation, we fix dimension $n=3$. 
\begin{lem}\label{transverse intersection is an open property}
	Let $s$ be a geodesic segment of length $2R\leq 2\inj(M)$. Then the open subset of $\grm$ whose points intersect the central half of $s$ transversally with an angle $> h_1(R)$ has Liouville measure $\geq h_2(R)$. 
\end{lem}
\begin{proof}
	Let $s=s(z)$ be a unit-speed parameterization where $z\in I'=[-R,R]$ and $s(I')$ is embedded. Let $I=[\frac{-R}{2}, \frac{R}{2}]$ parameterize the central half of $s$. Extend $z$ to the cylindrical coordinates $(r, \theta, z)$  centered at $s$, where $r, \theta$ form the $2$-dimensional polar coordinate for the hyperbolic disk orthogonal to $s$ at $s(z)$. The hyperbolic metric is given by $\diff r^2+\sinh^2 r \diff \theta^2 +\cosh^2 r \diff z^2$ (see [\citenum{bdNorms}, p. 548]), where $r=0$ corresponds to the geodesic $s$, and $r=\text{constant}$ gives concentric Euclidean annulus. Let $\phi_1, \phi_2$ be the spherical coordinates for the unit tangent space at $(r, \theta, z)$, where $\phi_1\in [0, \pi]$ measures the angle relative to $z$-axis and $\phi_2\in [0, 2\pi]$ corresponds to $\theta$. Hence 
	$(r, \theta, z, \phi_1, \phi_2)$ form local coordinates for $(x,v)$ in the unit tangent bundle $T^1M$ which double covers $\gtm$. 
	
	We now describe an open subset $O(s) \subset \gtm$ whose points are tangent to planes transversally intersecting $s(I)$ with angle $> h_1(R)$. Using the ball model, we place the segment $s$ at a vertical $\hp$ so that its perpendicular bisector passes through the origin. Then $s(I)=\{(0,0, z)|z\in [-R/2, R/2]\}$. The angle $\theta=0$ (resp. $\theta=\pi$) corresponds to the right (resp. left) half-space of the complement of the geodesic line containing $s(I)$ in $\hp$. The coordinates $r\in [0, R/2]$ and $\theta \in [0,2\pi]$ parameterize a solid cylinder $A$ centered at $s$. Let $\theta_0$ be the angle formed by the segment $\{(r, \pi, R/4)|r\in [0, R/2]\}$ and the segment $[(R/2, \pi, R/4), (0,0, R/2)]$. Combining [\citenum{buser2010geometry}, (i), (iii), (vi)], we deduce that for a right triangle with legs $a$ and $b$, opposite to angles $\alpha$ and $\beta$, respectively, we have 
	$$\frac{\cos \alpha}{\sin \alpha}=\frac{\sinh b \cosh a}{\sinh a}.$$	
	Set $a=R/4, b=R/2, \alpha=\theta_0$, we see 
	$$ \theta_0= \cot^{-1}\left(2\cosh^2 \frac{R}{4}\right)\tn{ and } \beta=\cot^{-1}\left(\frac{\sinh \frac{R}{4} \cosh \frac{R}{2}}{\sinh \frac{R}{2}}\right).$$
	 At $x=(r,\theta, z)\in A$, let $v$ be the unique unit vector orthogonal to a plane which is orthogonal to $s$. Let $O(\theta_0) \subset T^1_x M $ be vectors forming an angle with $v$ less than or equal to $\theta_0$. Then $O(\theta_0)$ is a disk on the unit sphere with measure $\frac{1-\cos\theta_0}{2}$. The plane  through $x$ orthogonal to $(\phi_1, \phi_2)=(\theta_0, \pi)$ intersects $s(I)$ at an angle $\beta$ and the plane orthogonal to $(\theta_0, \pi/2)$ intersects $s(I)$ at an angle $\pi/2-\theta_0$. We deduce that every plane orthogonal to a vector in $O(\theta_0)$ intersects $s(I)$ at an angle 
	 $$> \cot^{-1}\left(\frac{\sinh \frac{R}{4} \cosh \frac{R}{2}}{\sinh \frac{R}{2}} \right) -\cot^{-1}\left(2\cosh^2 \frac{R}{4} \right).$$
	 Thus $O(s)$ has measure $\geq$ 
	$$\ds \int_{0}^{R/2}\int_{0}^{2\pi}\int_{-R/4}^{R/4}(1-\cos\theta_0) \sinh r \cosh r \diff z \diff \theta \diff r =\frac{\pi(1-\cos\theta_0)}{4}(R\cosh R-R). $$ 
\end{proof}
\noindent
Using the metric on the Grassmann bundle (\Cref{grassmann metric}), any $(x, \Pi) \in O(s)$ is $\frac{\theta_0}{\sqrt{2\pi}}$ close to (one which is tangent to a plane) orthogonal to $s$.  

The proof of \Cref{large Fuchsian component}.
\begin{proof}
	We give a proof for $n=3$ and for higher dimensions the proof works similarly. 
	We first prove that the maximal diameter of components of $M-S_i$ tends to $0$ as $g_i\coloneqq \genus[S_i]\rightarrow \infty$. Suppose, for contradiction, that there exists a subsequence $S_i$ with a component $C_i \subset M-S_i$ such that $ \diam(C_i) \geq 2R >0$ (we may assume $R<\inj(M)$).
	
 Since $S_i$ is totally geodesic, each component of $\hy-\pi^{-1}(S_i)$ is the intersection of finitely many geodesic half-spaces and thus must be convex. In particular, $C_i$ contains a geodesic segment $\gamma_i$ of length $2R$. Since $M$ is compact, up to taking a subsequence, $\gamma_i$ converges to a geodesic segment $\gamma$ of length $2R$. Using \Cref{transverse intersection is an open property}, the subset of Grassmann bundles tangent to planes transversally intersecting the central half of $\gamma$ with an angle $\geq h_1(R)$ contains an open set $O$ of measure $\mul(O)\geq h_2(R)>0$. 
 
 By \cite{msErgodicInvariantMeasure}, the surfaces $S_i$ become equidistributed, i.e., for $\mu_i \coloneqq \mu_{S_i}$
$$
 	\mu_i \rightharpoonup \mu_{\liouville}
$$
 in the weak$-^*$ topology, where $\mul$ is the Liouville measure supported on $\gtm$. Thus $$\liminf\mu_i(O)\geq \mul(O)\geq h_2(R)>0. $$ 
 This implies that for sufficiently large $i$, $S_i$ intersects the central half of $\gamma$ with an angle $\geq h_1(R)$. Combining this with the Hausdorff convergence $\gamma_i \rightarrow\gamma$, we conclude that $S_i$ must intersect $\gamma_i$ transversally. Contradiction. 
 
 Since the maximal diameter of components of $M-S_i$ tends to $0$, their maximal volume also tends to $0$. Thus the number of components necessarily tends to infinity. If a component of $M-S_i$ has diameter $<\inj(M)$, it is contractible. 	 
\end{proof}

Given a closed hyperbolic $n$-manifold with the side-pairings $\{g_1, \cdots, g_k\}$ of $\pi_1(M)\rightarrow \Gamma$ (in the sense of \Cref{fundamental domain convex with side-pairings}) for a fundamental domain $D$ with sides $F_i, F_i'$ (for $i=1, \cdots, k$), one can prove there exist $2k$ open subsets $O_i, O_i'$ of $\mathcal{G}_{n-1}M$ so that a \toge hypersurface $S$ whose $\mathcal{G}_{n-1}M $ passes through all $O_i$ and $O_i'$ are filling, using \Cref{transverse intersection is an open property}. Each $O_i$ (resp. $O_i'$) is a subset of $\mathcal{G}_{n-1}M$ tangent to planes nearly orthogonal to the axis of $g_i$ and based at points close to $F_i$ (resp. $F_i'$). 

Let $\plm$ be the set of planes. 

\begin{cor}\label{finitely many non filling}
	Let $M$ be a closed hyperbolic $n$-manifold. All but finitely many planes in $\plm$ are strongly filling. 
\end{cor}
\begin{proof}
	By Ratner's \cite{rmRatner'stheorem} and Shah's \cite{snRatnerTheorem} theorem, a plane corresponds to either a dense subset of $\mathcal{G}_{n-1}M$ or a closed \toge  hypersurface. The projection of a dense plane of $\mathcal{G}_{n-1}M$ must be strongly filling in $M$, and there are at most finitely many non-filling discrete planes in $M$ by \Cref{large Fuchsian component}. 
\end{proof}
\begin{rem}
\end{rem}
 A closed hyperbolic $3$-manifold lies in exactly one of the following categories:
\begin{enumerate}
	\item  arithmetic with infinitely many \cbility classes of Fuchsian surfaces (by \cite{raGeodesicSurfacesHyperbolic}), 
	\item arithmetic with no Fuchsian surfaces (by \cite{raGeodesicSurfacesHyperbolic}), or 
	\item non-arithmetic with at most finitely many \cbility classes of Fuchsian surfaces (by \cite{bfmsTotallyGeodesicSubmanifolds}).
\end{enumerate}
\Cref{large Fuchsian component} trivially holds for the second and third categories. [\citenum{mrAritheoremeticHyperbolic3Manifolds}, Theorem 9.5.4] shows how to construct infinitely many arithmetic hyperbolic $3$-manifolds containing Fuchsian surfaces. 
\subsection{Nearly geodesic surfaces are strongly filling}
\cite{afLimitsAsymptoticalFuchsian} and the proofs of \Cref{limit of almost geodesic surfaces converge to dense subset} and \Cref{nearly geodesic are filling} will reveal the subtle difference between measure convergence and Hausdorff convergence. It is illuminating to examine hyperbolic surfaces first. Let $S$ be a genus-$2$ hyperbolic surface and $\beta$ a simple closed geodesic. By $\db$ we mean applying the Dehn twist with respect to $\beta$. Let $\pml$ denote the space of projective measured laminations on $S$ and $C(S)$ denote the collection of closed sets in $S$ equipped with Hausdorff metric. From the basic theory of geodesic laminations (e.g., [\citenum{bmIntroductionGeometricTopology}, 8.3]), we have as $n\rightarrow \infty$, 
\begin{align*}
	\frac{1}{\ell(\dbo^n\dbt^n\gamma)}\dbo^n\dbt^n\gamma &\rightarrow \frac{1}{2}\beta_1 \sqcup \frac{1}{2}\beta_2 \tn{ in } \pml \tn{ with the weak-}^* \tn{ topology,} \\
\dbo^n\dbt^n\gamma &\rightarrow \lambda_1 \cup \beta_1 \cup \lambda_2\cup \beta_2 \tn{ in } C(S) \tn{ with the Hausdorff metric topology,}
\end{align*}
where $\lambda_i$ are simple complete geodesic spiraling around $\beta_1, \beta_2$. See \Cref{measure vs Hausdorff}. 
	\begin{figure}[H]
	\centering
	\begin{subfigure}{0.35\textwidth}
		\centering	\includegraphics[width=.8\linewidth]{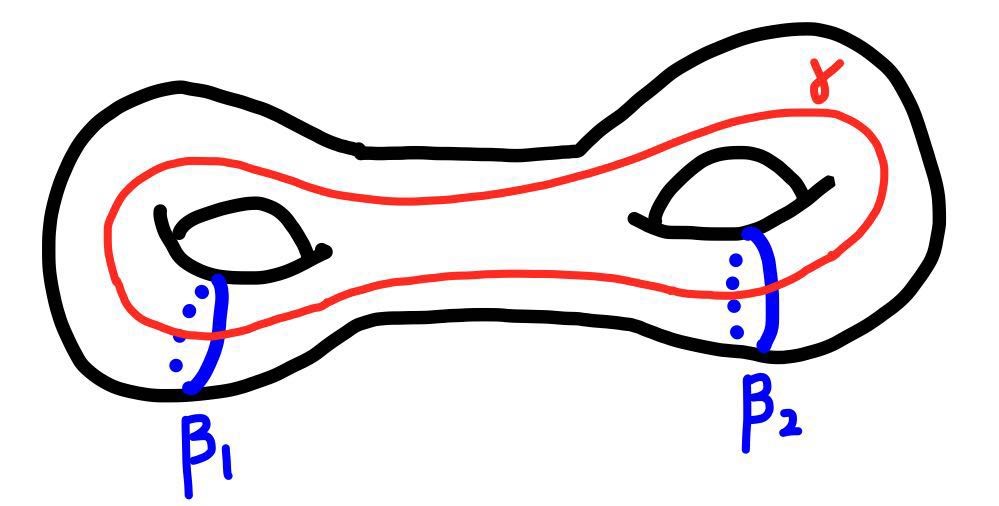}
		\caption{Applying Dehn twists to $\gamma$ with respect to $\beta_1, \beta_2$.}
	\end{subfigure}
	\hspace{0.1in}
	\begin{subfigure}{0.35\textwidth}
		\centering	\includegraphics[width=.8\linewidth]{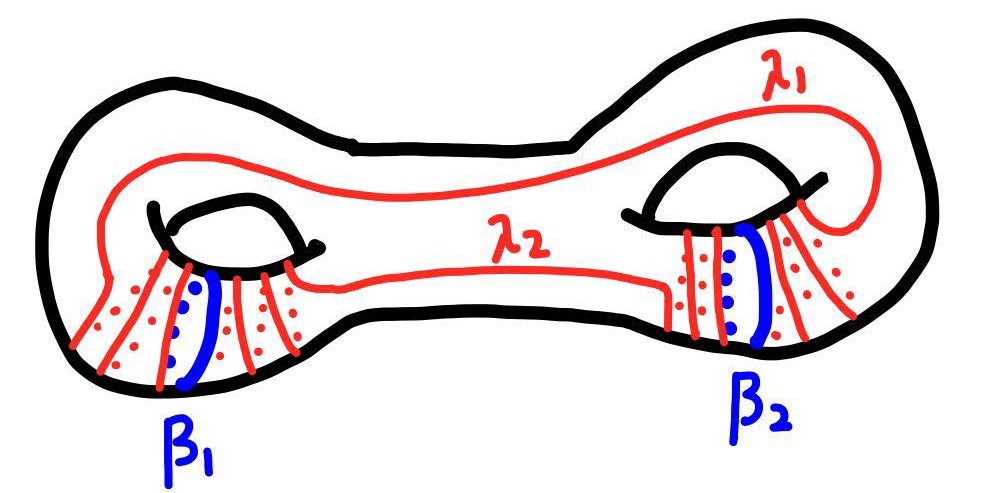}
		\caption{The Hausdorff limit carries a connected line field.}
	\end{subfigure}
	\caption{The weak-$^*$ limit is supported on two components. The Hausdorff limit connects the two components by two geodesics which are neither closed nor dense and forms a connected subset of the line bundle over $S$.}
	\label{measure vs Hausdorff}
\end{figure}
In contrast, every plane in a closed hyperbolic $3$-manifold $M$ is either dense  or a closed Fuchsian surface by Ratner's and Shah's theorem \cite{rmRatner'stheorem,snRatnerTheorem}. In Al Assal's thesis \cite{afLimitsAsymptoticalFuchsian}, he proves that the set of weak-$^*$ limits of the probability area measures induced on the Grassmann bundle by \ag  minimal surfaces achieves every convex combination of ergodic $\psltr$-invariant measures. When $M$ has no Fuchsian surfaces, the limiting measure must be the Liouville measure which implies that $\gt S_i$ equidistribute. When $M$ has a Fuchsian surface $S$, \Cref{limit of almost geodesic surfaces converge to dense subset} says that even if the limiting measure of a sequence of \ag surfaces $S_i$ is entirely concentrated on $S$, $\gt S_i$ is still asymptotically dense in $\gt M$, unless $S_i$ are eventually commensurable with $S$. If the limiting measure of $S_i$ is not entirely supported on a single Fuchsian surface, then $\gt S_i$ must be asymptotically dense. 

\begin{lem}\label{smoothly immersed minimal surfaces and connected Grassmann bundle}
 A $\pi_1$-injective surface $S$ in a $3$-manifold is homotopic to a smoothly immersed minimal surface whose Grassmann bundle is a connected closed subset of $\grm$. 
\end{lem}
\begin{proof}
	The existence of a smooth minimal surface representative follows from Schoen-Yau [\citenum{syExistIncomMinTopo}, Lemma 3.2 and Theorem 2.1] and Sacks-Uhlenbeck \cite{suMinimalImmersionClosedRiemann}. Osserman \cite{orNobranchedpoints} and Gulliver \cite{grRegularity} prove such minimal surfaces have no branched points. The smoothness of the tangent bundle of $S$ implies that $\gr S$ is a connected closed subset of $\grm$. The uniqueness of the minimal surface representative in \Cref{uniqueness} allows us to identify a homotopy class $[S]$ of $(1+\epsilon)$-\qf surfaces with a unique subset $\gt S$ of $\gtm$. 	
\end{proof}
\cite{afLimitsAsymptoticalFuchsian} proves that for a sequence of \ag surfaces $S_i$, both the pleated maps $f^p_i$ and minimal maps $f^m_i$ homotopic to $S_i\hookrightarrow M$ have the same weak-$^*$ measure limit along any convergent subsequence. 
The universal cover of $f^p_i$ may develop wrinkles and each $f^p_i$ lifts to a disconnected subset of $\gr M$. 

We now prove the following theorem which implies \Cref{nearly geodesic are filling}. 
\begin{theorem}
	Let $M$ be a closed hyperbolic $3$-manifold. If $S_i$ is a sequence of classes of \ag surfaces, then all but finitely many $[S_i]$ are \stf. 
\end{theorem}
\begin{proof}
We proceed by contradiction. Suppose $[S_i]$ is a sequence of $(1+\epsilon_i)$-\qf surfaces that are not \stf, where $\epsilon_i \rightarrow 0$. For each $[S_i]$, we take the minimal surface representative $S_i$. After taking a subsequence and renaming it if necessary, we have
\begin{lem}\label{psltr invariance of grs}
The Grassmann bundles $\gr S_i$ converge to a connected closed $\psltr$-invariant subset of $\grm$.
\end{lem}
\begin{proof}
Let $C(\grm)$ denote the set of closed subsets in $\grm$, equipped with Hausdorff metric. Since $M$ is closed, $\gtm$ is closed, and hence $C(\grm)$ is a compact metric space. Thus $\gr S_i$ has a convergent subsequence, which we still denote by $\gr S_i$. By \Cref{smoothly immersed minimal surfaces and connected Grassmann bundle}, each $\gt S_i$ is a connected closed subset of $\grm$, and thus
	$$\gr S_i \rightarrow \gr S \subset \grm$$
	where $\gr S$ is closed and connected. 
	
	The claim is that $\gr S$ is $\psltr$-invariant. Take $(x, \Pi)\in \gr S$ and let $P$ be the geodesic plane tangent to $(x, \Pi)$. Then $(x, \Pi)$ is the  limit of $(x_i, \Pi_i)\in \gt S_i$. Since $\epsilon_i \rightarrow 0$, the $L^{\infty}$-norms of the second fundamental forms of $S_i$ tend to $0$ by \Cref{seppi curvature estimates}. Lifting $S_i$ to $\tsi \subset \hy$ so that $\tsi$ all intersect a fixed fundamental domain.
	Since any loop in $\hs$ must have some points with curvature $>1$, $\tsi$ are eventually smoothly embedded disks with infinite injectivity radius.    
	By standard compactness theory for minimal surfaces (e.g., [\citenum{cmCourseMinSurface}, Proposition 7.14]), we have that the limit of the pointed Hausdorff convergence of $(\tilde{x_i}, \tilde{S_i})$ is the same as the graphical convergence on compact sets to a totally geodesic disc, after passing to a subsequence. Thus any disk $D(x, R) \subset P\subset \gt S$ of radius $R$ is the Hausdorff limit of disks $D_i(x_i, R_i)\subset \gr S_i$. Since $R$ is arbitrary, $P\subset \gr S$. 
\end{proof}
\noindent
Here $S$ is either one or uncountably many planes in $\hy$.
\begin{lem}\label{Hausdorff limit of \ag  surfaces}
	If the limit $\gr S$ is not a single discrete plane, it is dense and $\gt S_i \rightarrow \gt M$. 
\end{lem}
\begin{proof}
Enumerate the collection of discrete planes as $\{P_i\}$ which are embedded and mutually disjoint in $\grm$. Denote $\planes_j \coloneqq 
	\{P_i : i\leq j\}.$ 
By \Cref{psltr invariance of grs}, $\gt S$ is $\psltr$-invariant, and thus $\gt S\cap P_i$ is either $P_i$ or the empty set. 
For any finite $j$, the union $\cup_{\planes_j} P_i$ is a disconnected closed set in $\grm$. Hence $\gr S \cap \cup_{\planes_j} P_i$  is a strict subset of $\gr S$ for any $j$, unless $\gt S \subset P_i$ for some $i$, contradiction. 
 Thus for any genus $g$, there exist planes $P_k$ of genus $>g$ such that $P_k\subset\gr S -\cup_{\planes_j} P_i $ by \Cref{psltr invariance of grs}. 
  Since $P_k$ equidistribute as $k\rightarrow \infty$ by \cite{msErgodicInvariantMeasure}, $P_k \subset \gr S$ Hausdorff converge to a plane $P$ dense in $\gtm$. Since $\gr S$ is closed and $\psltr$-invariant, $P \subset \gr S$. 
\end{proof}

 Now suppose that $\gr S$ is supported on a single discrete plane $P$. Then for any small $\eta>0$, $\gr S_i \subset N_{\eta}(\gr S)$ for all sufficiently large $i$. 
 \begin{lem}\label{ag surfaces not supported in small nbhd of finitely many Fuchsian}
 	Let $S\subset M$ be totally geodesic and $\eta$ smaller than the normal injectivity radius of $\gt S$ in $\gtm$. If $S'$ is a smoothly immersed $\pi_1$-injective surface and $\gr S' \subset N_{\eta}(\gr S)$, then $S'$ and $S$ have homotopic images. 
 \end{lem}
 \begin{proof}
 	By a homotopy, $S'$ is minimal. Since $\gt S$ and $N_{\eta}(\gr S)$ are embedded, the preimages of $N_{\eta}(\gr S)$ in $\gt \hy$ consist of pairwise disjoint embedded neighborhoods $\{N_{\eta}( P) | \pr(P)\subset \hy \tn{ cover }S\}$ (where $P\subset \gt \hy$ is a plane). By \Cref{smoothly immersed minimal surfaces and connected Grassmann bundle}, $\gt S'$ is connected, and hence lifts to $\gt \ts'\subset N_{\eta}( P)$. The projection $\pr: \gr \hy\rightarrow \hy$ maps $\gt \ts'$ to $\ts' \subset N_{\eta}(\pr(P))$. The only complete minimal disk contained in  $N_{\eta}( \pr(P))$ is the plane $\pr(P)$ itself. Thus $\ts' =\pr(P)$ and $S'$ is \cble with $S$. 
 \end{proof}

 \begin{lem}\label{dense implies filling}
 	If a sequence of Grassmann bundles $\gt S_i$ over minimal surfaces converges to a dense subset of $\gt M$, then all but finitely many $S_i$ are strongly filling. 
 \end{lem}
 \begin{proof}
 	Suppose there exist nontrivial $[\gamma_i] \in \pi_1(M-S_i)$. Let $(x_i, v_i)$ be a unit tangent vector in the unit tangent bundle $T^1\gamma_i$. Since $T^1M$ is compact, there is a subsequence $(x_j, v_j)\rightarrow (x,v)$, which is tangent to a geodesic $\gamma$. By \Cref{Hausdorff limit of \ag  surfaces}, up to taking another subsequence, there exists $(p_j, \Pi_j) \in \gt S_j$ converging to $(x, \Pi)$ which is orthogonal to $(x,v)$. Looking at a fundamental domain $D$ in the universal cover $\hy$, we see there exist lifts $(\txj, \tvj)\rightarrow (\tx, \tv)$ in $D$, which are tangent to $\tgj$ and $\tg$, respectively. And there exist $(\tpj, \tpaij) \in \tsj \rightarrow (\tilde{x}, \tpai)$, which is tangent to a plane $\tp$. Recall $S_j$ are $(1+\epsilon_j)$-\qf and $\epsilon_j \rightarrow 0$. 
 	
 	As in the second paragraph of the proof of \Cref{psltr invariance of grs}, the limit of the pointed Hausdorff convergence of $(\tpj, \tsj)$ is the same as graphical convergence on compact sets to a totally geodesic disc, after passing to a subsequence. 
 	
 	By [\citenum{lvQuasiConformalMappings}, Theorems II.5.1 and II.5.3], a sequence of quasicircles with bounded quasi-Fuchsian constants has a subsequence converging in the
 	Hausdorff metric either to a quasicircle or to a point. 
 	Since the minimal surface $\tsj$ is contained in the convex core of $\partial \tsj$ (see [\citenum{saMinimalDiscsHyperbolicSpace}, Corollary 2.5]) and $(\tilde{p_j}, \widetilde{\Pi_j})\in \tsj$ converge to $ (\tilde{x}, \widetilde{\Pi})$, the Hausdorff limit of $\partial \tsj$ cannot be a point. Let $\tilde{\beta_j}$ be an intrinsic geodesic of $\tsj$ passing through $\tpj$. There exists a $\delta>0$ such that for all $j$ large enough, the intrinsic  curvature of $\tilde{\beta_j}$ is $<1-\delta$, since $S_j$ is \agn. By [\citenum{lcSmallCurvatureSurface}, Theorem 5.1], all $\beta_j$ are uniform quasi-geodesics of $\hy$. Thus as $(\tpj, \tsj)$ converge to $(\tx, \tpai)$ in the Hausdorff metric, $(\tpj, \tilde{\beta_j})$ converge to $(\tx, \tilde{\beta})$ a geodesic on $\tp$. Since $\tilde{\beta}_j$ are all uniform quasi-geodesics, their endpoints converge to the endpoints of $\tilde{\beta}$. Thus the limiting quasi-circle of $\partial \tsj$ must be the circle $\partial \tilde{P}$. Since $\tgj \rightarrow \tg$ which is orthogonal to $\tilde{P}$, for all large $j$,  $	\partial \tsj$ must separate $\partial\tgj$. This implies $\pi_1(M-S_j)$ cannot contain $\gamma_j$, which is a contradiction. 

 		We just proved there is a subsequence $S_j$ of the initial sequence $S_i$ which is strongly filling. If there exist infinitely many non-filling $S_i$, then we apply the above argument again to this collection and conclude it must contain a subsequence consisting of strongly filling surfaces, which is a contradiction.  
 \end{proof}
\end{proof}
\Cref{Hausdorff limit of \ag  surfaces} establishes \Cref{limit of almost geodesic surfaces converge to dense subset}. Lowe in [\citenum{lbDeformationsTotallyGeodesic}, Theorem 4.2] proves that the \ag  surfaces are dense in $\gtm$ assuming they have a limiting circle which is not the limit set of a closed Fuchsian surface.  Al Assal and Lowe \cite{alAsymptoticallyGeodesicsurfaces} independently obtain \Cref{limit of almost geodesic surfaces converge to dense subset} for finite-volume hyperbolic manifolds (including the noncompact case) and obtain various applications, including rigidity results for negatively curved manifolds [\citenum{alAsymptoticallyGeodesicsurfaces}, Theorem 1.7] and a gap theorem [\citenum{alAsymptoticallyGeodesicsurfaces}, Theorem 1.6] for the \qf constants of minimal surfaces in geometrically finite hyperbolic manifolds of infinite volume. 
\begin{rem}\label{connectedsurfaceUseAllpants}
Al Assal in \cite{afLimitsAsymptoticalFuchsian} constructs \ag surfaces which equidistribute or whose limit has a nontrivial scarring component. Building upon Liu-Markovic \cite{lmHomologyCurves}, in [\citenum{afLimitsAsymptoticalFuchsian}, Section 4, 5, and 7], in order to construct \emph{connected} \ag surfaces with various limiting properties, Al Assal uses \emph{all} $(R,\epsilon)$-good pants which are known to equidistribute (different limiting measures correspond to different weights on the set of good pants). Thus \Cref{Hausdorff limit of \ag  surfaces} and \cite{alAsymptoticallyGeodesicsurfaces} imply as $\epsilon$ gets smaller, the pants \emph{required} to construct a connected \eqf surface become dense.  
\end{rem}
\kms \cite{kmsGeometricallyTopologicallyRandomsurface} study the average limiting measures supported by the nearly geodesics surfaces. They discover that the distribution of random surfaces with bounded genus is drastically different from that of random surfaces with bounded area. 

The proof of \Cref{limit of almost geodesic surfaces converge to dense subset}  implies that the boundary circles of \ag surfaces in $\partial \hy$ forms a ubiquitous set. 

The proof of \Cref{ubiquitousSurfaces} (ubiquity of the boundary circles of \ag surfaces). 
\begin{proof}
	The idea is essentially in the proof of \Cref{dense implies filling}. 
	Using \Cref{limit of almost geodesic surfaces converge to dense subset}, the Grassmann bundles $\gt S_i$ converge to $\gt M$. Let $C_1, C_2$ be two disjoint and non tangent circles at $\partial \tm$ which bound two planes $P_1, P_2\subset \hy$, respectively. We want to show there exists some $S_i$ which has a preimage $\tsi$ whose limit set separates $C_1$ and $C_2$.
	
	Let $(\tilde{p}, \widetilde{\Pi})$ be an element of $\gt \tm$ which is tangent to a plane $P$ such that $ P$ separates $P_1$ and $ P_2$. Let $(p, \Pi)=\pi(\tilde{p}, \widetilde{\Pi})$ be a element in $\gt M$. By the convergence $\gt S_i \rightarrow \gt M$, there exist $(p_i, \Pi_i)\in \gt S_i$ which converge to $(p, \Pi)$. In the universal cover, there exists a preimage $\gt\widetilde{S_i}$ containing $(\tilde{p_i}, \widetilde{\Pi_i})$ such that $$(\tilde{p_i}, \widetilde{\Pi_i})\rightarrow (\tilde{p}, \widetilde{\Pi}).$$ 
	
	Then the second paragraph of the proof of \Cref{dense implies filling} implies that for sufficiently large $i$, $\partial \tsi$ separates $C_1$ and $C_2$ and $\tsi$ separate $P_1$ and $P_2$.  
\end{proof}
\subsection{The dichotomy of embedded surfaces in a hyperbolic $3$-manifold}
In this subsection, we prove the dichotomy for embedded surfaces in \Cref{embedded minimal surfaces of sufficiently small curvature is geodesic}. The following lemma is due to Feustel [\citenum{fcSurfacesIrreducible}, Theorem 1], Jaco [\citenum{jwFinitelyPresentedSubgroups}, Theorem 5], and Scott [\citenum{spSufficientlyLarge}, Theorem 1.4]. We present a concise proof with more modern terminology for the convenience of the readers. 
\begin{lem}\label{nontrivial covering of orientable is not embedded}
	Let $M$ be a compact $3$-manifold. If $S$ is a closed $\pi_1$-injective surface which nontrivially covers another orientable $\pi_1$-injective surface $S'$ embedded in $M$, then $S$ is not embedded. 
\end{lem}
\begin{proof}
	The $\pi_1$-injectivity of $S$ and $S'$ in $M$ induces an injection of the fundamental group: 
	$$\pi(S) \rightarrow \pi_1(S') \rightarrow \pi_1(M).$$ 
	Suppose $S$ is embedded in $M$. Then, by lifting to the covering $M'$ of $M$ corresponding to $\pi_1(S')$, we also lift $S$ homeomorphically to an embedded \poi  surface $\tilde{S}$ in $M'$. Since $\pi_1(M') \cong \pi_1(S')$, by the classification of $3$-manifolds whose fundamental groups are virtually surface subgroups ([\citenum{hj3-manifolds}, 10.6 Theorem]), we have $M' \cong S'\times I$. By the classification of \poi  embedded surfaces in an $I$-bundle [\citenum{bmIntroductionGeometricTopology}, Proposition 9.3.18], $\tilde{S}$ is isotopic to the horizontal surface $S' \times 0$. This contradicts the assumption that $\tilde{S}$ is a nontrivial covering of $S'$. 
\end{proof}

Proof of \Cref{embedded minimal surfaces of sufficiently small curvature is geodesic} (on the dichotomy of embedded \poi surfaces in $M$). 
\begin{proof}
	Let $S$ be an embedded surface in $M$. By \Cref{nontrivial covering of orientable is not embedded}, $S$ is a primitive orientable surface.
	By the Lefschetz duality (e.g., [\citenum{bmIntroductionGeometricTopology}, Corollary 9.1.5]), we have 
	$\dim H^1(M-S; \Q) >0$
	which implies that $S$ is never filling.
	By \Cref{nearly geodesic are filling}, there exists an $\epsilon_0>0$ such that every $(1+\epsilon)$-\qf surface with $0<\epsilon\leq \epsilon_0$ is filling. Thus an embedded non Fuchsian surface satisfies $\epsilon>\epsilon_0$. 
	
	Now let $S\hookrightarrow M$ be an embedded minimal surface of supremum curvature $\lambda_0 \leq \frac{\epsilon_0}{2+\epsilon_0}<1$. We want to show $S$ is totally geodesic of genus $<\gn $. By Uhlenbeck \cite{ukclosedminSurfaceHyperbolic}, having supremum curvature $<1$ implies that $S$ is almost-Fuchsian, say $(1+\epsilon)$-quasi-Fuchsian. 
	By \Cref{huangwangCurvatureupperboundQF},  
	$$1+\epsilon\leq \frac{1+\lambda_0}{1-\lambda_0}\leq 1+\epsilon_0. $$
	If $\epsilon \neq 0$, then \Cref{nearly geodesic are filling} forces $S$ to be filling and hence non embedded. Thus $\epsilon=0$ and $S$ is Fuchsian. Suppose the genus $\gs \geq \gn $. If $S$ is an orientable surface attaining the smallest genus in its \cbility class, then \Cref{nearly geodesic are filling} implies $S$ is filling. If $S$ is \cble with a Fuchsian surface of genus $< \gn $, since commensurable Fuchsian surfaces span the same subset of $M$ (i.e., geometrically identical or same $\psltr$-orbit), the immersion $S\hookrightarrow M$ maps some points in $S$ to points in $M$ with multiplicity more than $1$. This violates the assumption that $S$ is embedded. Thus $S$ must have genus $<\gn $. 
\end{proof}
The following example from Dunfield gives infinitely many closed hyperbolic $3$-manifolds with infinitely many embedded \qf surfaces.  
\begin{exa}
	\label{example of closed hyp 3mf with infinite embedded QF surfaces}
	Let $K$ be a Montesinos knot with at least four rational tangles with all $q_i  \geq 3$.  By [\citenum{ouClosedIncompressibleStar}, Corollary 3], $M = S^3 - K$ contains embedded \poi surfaces of every genus $> 1$.  Let $Y_n$ be $1/n$ Dehn surgery on $M$, and pick $n $ sufficiently large so that $Y_n$ is closed hyperbolic by Thurston's Dehn surgery theorem (e.g., [\citenum{bmIntroductionGeometricTopology}, Corollary 15.1.3]).  By the proof of [\citenum{ouClosedIncompressibleStar}, Corollary 4(b)], all the surfaces in $M$ remain essential in $Y_n$.  Since $H_1(Y_n; \Z)$ = 0, none of these surfaces can be fibers or semi-fibers. Bonahon-Thurston's theorem \cite{twGeoTop,bfBoutsDesVar} says that \poi surfaces in $M$ is either \qf or virtually fibered.  Hence, $Y_n$ is a closed hyperbolic $3$-manifold with essential embedded \qf surfaces of every genus $> 1$.
\end{exa}

[\citenum{myVolumeHypGeodesicBoundary}, Theorem 4.2] bounds the number of embedded Fuchsian surfaces in $M$ by the volume of $M$ up to a uniform constant. 
Let $\me_0$ be the set of embedded Fuchsian surfaces and $\me_{\geq \epsilon_0}$ the set of \qf surfaces with constant at least $1+\epsilon_0$ but not any lower. 
The deviation from being totally geodesic can also be measured  geometrically using a uniform lower bound on the supremum principal curvature $\lambda_0$, i.e., 
$$ \inf\limits_{S\in \me_{\geq \epsilon_0}}\lambda_0(S)\geq \delta_1$$
when we take $S$ to be minimal in its homotopy class. The principal curvature lower bound uses [\citenum{ecHyperbolicGauss}, (5.5)]. A fibered surface is uniformly away from being totally geodesic by considering either the supremum principal curvature \cite{ukclosedminSurfaceHyperbolic} or the Hausdorff dimension of its limit set. 

\noindent In general, it is rare that a hyperbolic $3$-manifold contains a Fuchsian surface. Basilio, Lee, and Malionek \cite{blmGeodesicSurfaces} develop an algorithm that detects whether an ideally triangulated  $3$-manifold contains Fuchsian surfaces: out of 150,000 manifolds, they find $9$.

\subsection{Least area non-Fuchsian surfaces fail to minimize intersection numbers}
\begin{lem}\label{non geodesic least area surface does not minimize intersection number}
	Let $M$ be a hyperbolic $3$-manifold. Then for any \las non-Fuchsian surface $S$, there exists a closed geodesic $\gamma$ such that the intersection number $|S\cap \gamma|$ is strictly larger than the geometric intersection number $[S]\cdot [\gamma]$.
\end{lem}
This contrasts with \Cref{intersection between geodesic and codimension one submanifold} and the intersection number minimizer on a hyperbolic surface.
\begin{proof}
	Take any non Fuchsian \las surface $S$ of genus $g$ in $M$. Since minimal surfaces in $3$-manifolds are smoothly immersed by \Cref{smoothly immersed minimal surfaces and connected Grassmann bundle} and non transverse self-intersections are isolated by [\citenum{fhsLeastAreaImmersion}, Lemma 1.4], there is a point $p\in S$ not in the self-intersection locus of $S$. There exists a tangent plane at $p$ and also a plane $\Pi\subset M$ tangent to $S$ at $p$. 
	
	By [\citenum{cmCourseMinSurface}, Theorem 7.3, Lemma 6.13] or [\citenum{fhsLeastAreaImmersion}, Lemma 1.4], the local geometry of the minimal surface intersecting $\Pi$ can be described using classical models of non transverse intersection, which looks like $z=\tn{Re}(x+iy)^n$ intersecting $z=0$, resulting an $n$-prongs. Near $p$, there are some numbers $>1$ of parts of the minimal surface that are larger than $z=0$ (the tangent geodesic plane). 
	\begin{figure}[h]
		\centering
		\includegraphics[width=0.4\linewidth]{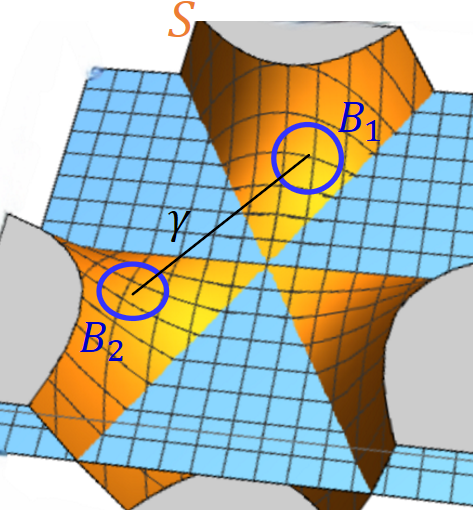}
		\caption{The \las $S$ has redundant intersections with some geodesic $\gamma$}
		\label{Non essential intersections between closed geodesic and least area surfaces}
	\end{figure}
	We can draw two small open balls $B_1, B_2$ supported in a small neighborhood $U$ of $p$, so that $U$ is away from the self-intersection locus. See \Cref{Non essential intersections between closed geodesic and least area surfaces}. The set of geodesics which pass through $B_1$ and $B_2$ in $M$ contains an open subset. By [\citenum{khModernDynamics}, Theorem 17.6.2, Corollary 6.4.20], the collection of closed geodesics is a dense subset of the unit tangent bundle of $M$. Hence, there is a closed geodesic $\gamma$ in $M$ which passes through $B_1$ and $B_2$. Then the geodesic $\gamma$ intersects the least area surface $S$ at two points which are non essential intersections, and can be removed after applying a homotopy to $\gamma$. 
\end{proof}
 
 \subsection{Cubulations and one-orbit of hyperplanes}\label{cubulations}
 For backgrounds and basic definitions of cube complex and its connection to $3$-manifolds, we refer the readers to  \cite{afw3manifoldGroups,fhDeformingCubulations,bwBoundaryCriterionCubulation} and the references therein.
 
 A \emph{cube complex} is a metric polyhedral complex all of whose cells are unit cubes, i.e., it is the quotient of a disjoint union of copies of unit cubes under an equivalence relation generated by a set of isometric identifications of faces of cubes.
 Throughout this subsection, let $X$ be a CAT(0) cube complex. We define an equivalence relation on the edges of $X$ generated by identifying opposite edges of some square in $X$. Given an equivalence class $[e]$ of edges, the hyperplane $w$ dual to $[e]$ is the collection of midcubes which intersect edges in $[e]$. The complement $X- w$ consists of two halfspaces. 
 
 Let $G$ be a finitely generated group with a Cayley graph $\Delta$. A subgroup $H\leq G$ is \textit{codimension}-$1$ if it has a finite neighborhood $N_r(H)$ such that $\Delta-N_r(H)$ contains at least two $H$-orbits of \textit{deep} components, which do not lie in any $N_s(H)$. For $M$ a closed hyperbolic $3$-manifold and $G=\pi_1(M)$, a \qf surface subgroup of $G$ is codimension-$1$ and quasiconvex.
 
 A CAT(0) cube complex $X$ is \textit{essential}, if for each hyperplane $w$, each of the associated halfspaces contains points in $X$ arbitrarily far from $w$. If a group $G$ acts by automorphism on $X$, the action is \textit{essential} if for any point $x$ in the zero-skeleton of $X$, and each hyperplane $w$, each of the associated halfspaces contains points in $G\cdot x$ arbitrarily far from $w$. 
 
 The cube complex $X$ is \textit{hyperplane-essential} if each hyperplane $w$, regarded itself as a CAT(0) cube complex, is essential. The action of $G$ on $X$ is \textit{hyperplane-essential} if each hyperplane $w$ has  the property that the stabilizer of $w$ acts essentially on $w$.
 
 If $G$ is a finitely generated group with a finite collection of codimension-$1$ subgroups $H_1,\cdots, H_k$, Sageev \cite{smCubeComplex} introduced a powerful construction that produces an action of $G$ on a CAT(0) cube complex $X$ that is dual to a system of walls corresponding to these subgroups. This elucidates many structures of the group. Moreover, he established the following useful criterion for compactness property. 
 \begin{theorem}[Sageev, \citenum{smCodim1Subgroups}]\label{cube complex cocompact}
 	Let $G$ be a word-hyperbolic group, and $H_1,...,H_k$ be a  	collection of quasiconvex codimension-$1$ subgroups. Then the action of $G$ on the 	dual cube complex is cocompact.
 \end{theorem} 
 
 Proof of \Cref{cor cubulations}.
 \begin{proof}
  
 	Since $M$ is closed hyperbolic, $\pi_1(M)$ is word-hyperbolic. 
 	Since a \stfs \qf surface (or \toge hypersurface) $[S]$ separates any two points on the Gromov boundary of $G$, the singleton $\{[S]\}$ satisfies the assumption of [\citenum{bwBoundaryCriterionCubulation}, Theorem 1.4]. Thus $G$ acts properly and cocompactly on the dual \catzeros cube complex $X$. Since $[S]$ is quasi-convex codimension-$1$, it follows that $X$ is essential. 
 	
 	The orbits of the hyperplanes correspond to the image $H$ under the natural action by $\pi_1(M)$. The stabilizers of the hyperplanes correspond to conjugation of $H \cong \pi_1(S)$. Since $S$ is compact, $\pi_1(S)$ is again a word-hyperbolic group. With the existence of filling closed geodesics, using similar reasoning above, we deduce that $X$ is hyperplane-essential. 
 	
 	Moreover, $[S]$ strongly filling corresponds to the dual cube complex having one-orbit of hyperplanes. 
 	
 	Since non commensurable strongly filling surfaces have different limit sets on $\partial G$, they give rise to cubulations that are not $\pi_1(M)$-equivariantly isometric. 
 	As a consequence of \Cref{nearly geodesic are filling} and \Cref{large Fuchsian component}, we establish \Cref{cor cubulations}.  
 \end{proof}
 This partially answers a question of Fioravanti-Hagen on [\citenum{fhDeformingCubulations}, p. 879]. For any Gromov-hyperbolic group $G$ that admits a proper, cocompact action on a CAT(0) cube complex,  Fioravanti-Hagen [\citenum{fhDeformingCubulations}] establish the existence of a \catzeros cube complex $X$ on which $G$ acts properly, cocompactly, and essentially with a single orbit of hyperplanes, or acts properly, cocompactly, essentially, and hyperplane-essentially. Moreover, if we apply their theorems to the case where $G=\pi_1(M)$ for $M$ a closed hyperbolic $3$-manifold, it is not clear that their hyperplane-stabilizer is a surface subgroup\footnote{M. Hagen informed the author that their construction applied to a filling collection of Fuchsian surfaces should give hyperplane-stabilizers that are surface subgroups.}. This is because their cubulation comes from taking a filling collection $A$ of finitely many surfaces constructed by \cite{kmImmersingAlmostGeodesic} and then applying the boundary criterion \cite{bwBoundaryCriterionCubulation}. Surfaces in $A$ may coincide on pants instead of just intersecting at closed curves. Thus it is not clear their hyperplane-stabilizers are surface subgroups. In Kahn-Wright's construction \cite{kwNearlyFuchsian}, one uses two copies of each good pants. See also \Cref{connectedsurfaceUseAllpants}. 
 
 \begin{rem}[An analogy between Fioravanti-Hagen's remark and our results on ubiquity]
 	 Fioravanti-Hagen remark on [\citenum{fhDeformingCubulations}, p.880] that for certain $G$ and $X$ (including our cases), an arbitrarily small deformation of one of the original walls suffices to obtain a proper action with a single orbit of hyperplanes. In particular, \ag surfaces (resp. totally geodesic hypersurfaces) are ubiquitous by \Cref{ubiquitousSurfaces} and eventually strongly filling \Cref{nearly geodesic are filling,large Fuchsian component}; thus in an arbitrarily small neighborhood of a geometric circle (resp. geometric sphere $S^{n-2}$) at $\partial \tm$, there exists a limit set of a \stfs surface (resp. hypersurface). The technique of \cite{fhDeformingCubulations} is purely combinatorial. 
 \end{rem}
 
	\section{From filling surfaces to the topology and geometry of the ambient manifold}\label{other}
Dunfield first suggested the existence of a filling \poi surfaces in non hyperbolic manifolds, such as the Hantzsche-Wendt manifold. In this section, we will classify closed orientable geometric $3$-manifolds which contain a filling \poi surface (\Cref{filling 3-manifold}) and prove the isoperimetric inequality between a filling surface and $M$ (\Cref{isoperimetric inequality filling surface has area at least volume}). 
 	\subsection{The proof of \Cref{filling 3-manifold}}\label{non hyperbolic filling}
  By \Cref{covering preserves filling surfaces}, we can pass to the maximal manifolds to investigate the existence of filling surfaces. We use [\citenum{afw3manifoldGroups}, Table 1] for model geometric $3$-manifolds which are maximal among compact quotients. \\
Proof of \Cref{filling 3-manifold}. 
\begin{proof}
The spherical manifold $S^3$ contains a great sphere which is filling. Similarly, $\R P^3$ contains $\R P^2$ as a filling \poi surface with nontrivial $\pi_1$. 
	\begin{lem}
		A closed Sol $3$-manifold has no filling surfaces. 
	\end{lem}
	\begin{proof}
		\cite{maHomomorphismFiniteGroups} proves that Sol $3$-manifolds are LERF (locally extended residually finite). If $S \hookrightarrow M$ is a \poi filling immersed surface, by taking a finite cover $M'$ of $M$ (which is still a torus bundle over the circle), we lift $S$ to a filling union of embedded surfaces $S' \subset M'$ by \Cref{covering preserves filling surfaces}. In a torus bundle with Anosov monodromy, if there exists an embedded \poi surface not homotopic to a fiber, its intersection with a fiber is a simple closed curve, which contradicts the fact the monodromy is Anosov\footnote{We thank Yanqing Zou for helpful comments which simplify the proof.}. Thus the only embedded essential surfaces in $M'$ consist of fibers, which are all homotopic and whose complements are of the form $\T^2 \times I$. 
	\end{proof}
	
	The remaining cases are various Seifert fibered geometries. If $F$ is an immersed surface in $M$, we call it \textit{horizontal} if $F$ is everywhere transverse to the fibers of $M$ and \textit{vertical} if $F$ is everywhere tangent to the fibers of $M$. Any $\pi_1$-injective surface is homotopic to a minimal surface by \cite{syExistIncomMinTopo,suMinimalImmersionClosedRiemann}. 
	Hass in \cite{hjMinimalsurfacesSeifertFiber} classifies immersed $\pi_1$-injective surfaces in Seifert fiber spaces using minimal surface theory. 
	\begin{theorem}[Hass]
		Let $M$ be a Seifert fiber space with a geometric structure. Let $S$ be a closed surface not equal to $S^2, \R P^2$, and $f: S\rightarrow M$ a minimal $\pi_1$-injective immersion. Then $F$ is either vertical or horizontal. 
	\end{theorem}
	A vertical surface is never filling, since the complement in $M$ contains $S^1$ fibers. By [\citenum{bmIntroductionGeometricTopology}, Proposition 10.4.8], if a closed Seifert fiber manifold contains a horizontal surface, its Euler class number must be $0$. This rules out the Nil and the $\widetilde{\tn{SL}(2,\R)}$ geometry.
	
	Scott \cite{spSubgroupsGeometric} shows that the fundamental group of a Seifert fibered manifold $M$ is LERF. Thus to understand the complementary regions to an immersed $\pi_1$-injective surface $S$ in $M$, we can pass to a finite cover $N$ of $M$ so that the preimage of $S$ in $N$ is a finite union of embedded \poi  surfaces that are invariant under the covering group $C$. \Cref{covering preserves filling surfaces} implies that it suffices to show that no such collection of embedded \poi  surfaces fills $N$. 
	
	A closed $3$-manifold with $S^2\times \R$ geometry is finitely covered by $S^2\times S^1$. Then $S^2 \times \theta $ is the only \poi  surface up to homotopy, and any finite union of these surfaces is non-filling. 
	
	By [\citenum{bmIntroductionGeometricTopology}, 12.3], the $6$ orientable closed Euclidean $3$-manifolds are
	\begin{enumerate}
		\item the $3$-torus $\T$ made by gluing opposite faces of a cube,
		\item the manifold made by gluing opposite faces of a cube with a 1/2 twist on one pair,
		\item the manifold made by gluing opposite faces of a cube with a 1/4 twist on one pair,
		\item the manifold made by gluing opposite faces of a hexagonal prism with a $1/3$ twist on the hexagonal faces,
		\item the manifold made by gluing opposite faces of a hexagonal prism with a $1/6$ twist on the hexagonal faces, and 
		\item the Hantzsche–Wendt manifold, whose fundamental domain (in the lower right-hand corner of \Cref{fundamental domains}) is made from two unit cubes whose faces are glued according to the markings (the unmarked top and bottom are glued with translations).  
	\end{enumerate}
	We denote these manifolds by $M_1, \cdots, M_6$. 
	Their fundamental domains are illustrated in \Cref{fundamental domains}\footnote{This figure is taken from [\citenum{bmIntroductionGeometricTopology}, Figure 12.2], courtesy of Bruno Martelli.}.
	\begin{figure}[h]
		\centering
		\includegraphics[scale=0.4]{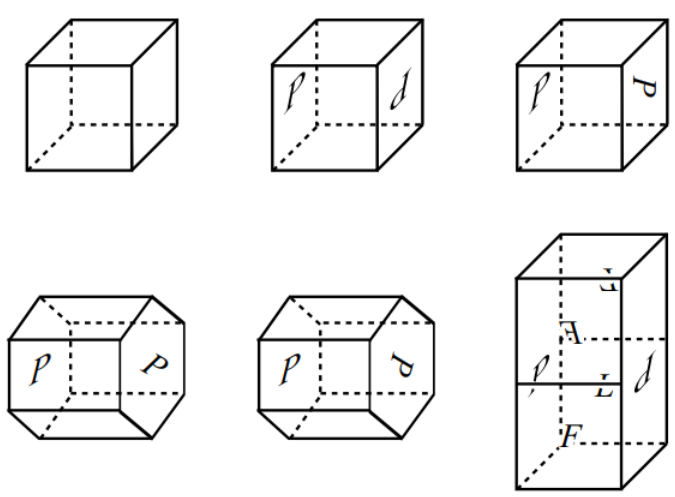}
		\caption{Fundamental domains of the six closed flat orientable $3$-manifolds. Unlabeled faces are glued by translations.}
		\label{fundamental domains}
	\end{figure}
	
	The $3$-torus has no filling surfaces, since every horizontal surface is a torus whose complement is $\T^2\times I$. 
	Using ideas in \Cref{filling fuchsian surface intersects every geodesic} and \Cref{intersection between geodesic and codimension one submanifold}, we prove
	\begin{lem}\label{three families to be filling}
		Let $M$ be a closed Euclidean $3$-manifold and $S$ a totally geodesic surface. Then the following are equivalent. 
		\begin{enumerate}
			\item The homotopy class $[S]$ is filling.
			\item The preimage $\ts \subset \R^3$ contains three families of planes whose normal vectors are linearly independent. 
		\end{enumerate}	 
	\end{lem}
	\begin{proof}
		$\Leftarrow$. Suppose $S$ is a geodesic surface such that $\ts$ contains three families of planes with linearly independent normals. Let $\tg$ be a geodesic in $\R^3$. Then $\tg$ transversely intersects some plane $\tp$ at a unique point. No homotopy applied to $\pi(\tg)$ and $\pi(\tp)$ in $M$ removes this intersection. This implies that 
	$\pi_1(M-S)$ cannot contain $[\pi(\tg)]$. Thus $[S]$ is filling. 
	
	$\Rightarrow$. Suppose $\ts$ only contains two families of planes. Then there exists a family of geodesics $ \subset \R^3-\ts$. Since $\ts$ is discrete, any geodesic $\tg$ in this family has a discrete $\Z^3$-orbits in $\R^3$. Thus $\pi(\tg)$ is a closed geodesic disjoint from $S$ in $M$. The remaining cases are similar. 
	\end{proof}
	Readers can compare \Cref{three families to be filling} with \Cref{filling fuchsian surface intersects every geodesic}. 
	
	For $i=2,3,6$, $C_i$ denotes the covering groups corresponding to the covering $M_1 \rightarrow M_i$. For $i=4,5$, $C_i$ refer to the groups corresponding to the covering from hexagonal $\T^3$ (the hexagonal fundamental domain with opposite faces glued to each other) to $M_i$. 
	
	The covering group $C_2$ of $M_2$ is of order $2$. The image of a torus in $\T$ under $C_2$ yields at most two families of planes in $\R^3$. 
	
	The covering groups of $M_3, \cdots, M_6$ have order $\geq 3$; hence, it is not surprising that they all have filling surfaces, in view of \Cref{three families to be filling}. 
	
	The covering group $C_3$ of $M_3$ is of order $4$. A filling surface corresponds to the image of $C_3$ acting on the diagonal torus from the upper left edge to the lower right edge in the front face (\Cref{m3}). 
	
	For $M_4$, the covering group $C_4$ is of order $6$ and acts on the hexagonal prism (after pulling back by translation) as rotation by $60$ degrees. A filling surface corresponds to the diagonal parallelogram annulus shown in the picture \Cref{m4}. The fiber directions are perpendicular to the hexagonal faces. The projection of a single annulus to the hexagonal face is a diagonal rectangle. Six of them obstruct all fibers by \Cref{three families to be filling}. The intersection of each hexagon with the six annuli consists of six arcs which cut the hexagon torus into contractible polygons. Thus the complement of the surface in $M_4$ is contractible. 
	\begin{figure}[h]
		\centering
			\begin{subfigure}{0.3\textwidth}
			\centering	\includegraphics[width=.7\linewidth]{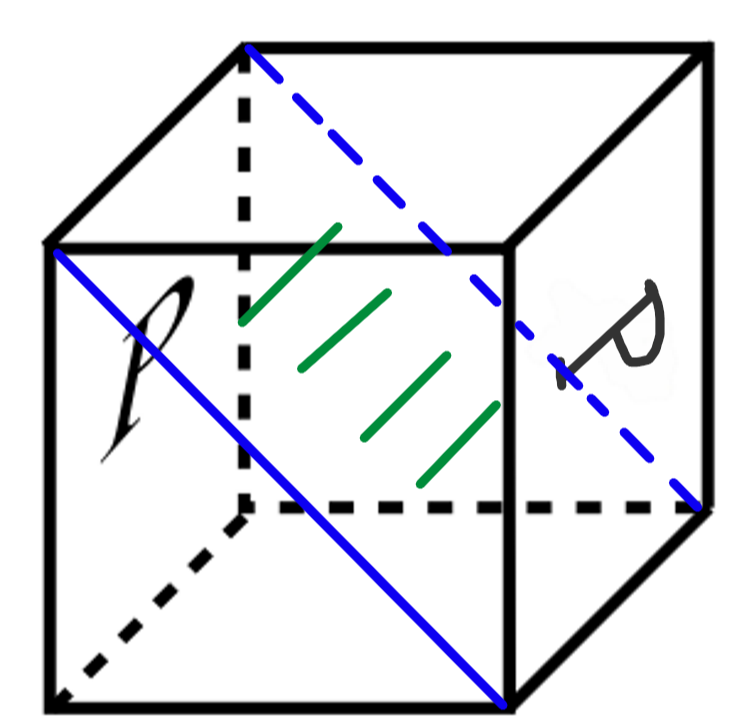}
			\caption{The annulus under $90^{\circ}$ rotation generates a filling surface in $M_3$}
			\label{m3}
		\end{subfigure}
		\hspace{0.1in}
		\begin{subfigure}{0.3\textwidth}
			\centering	\includegraphics[width=.7\linewidth]{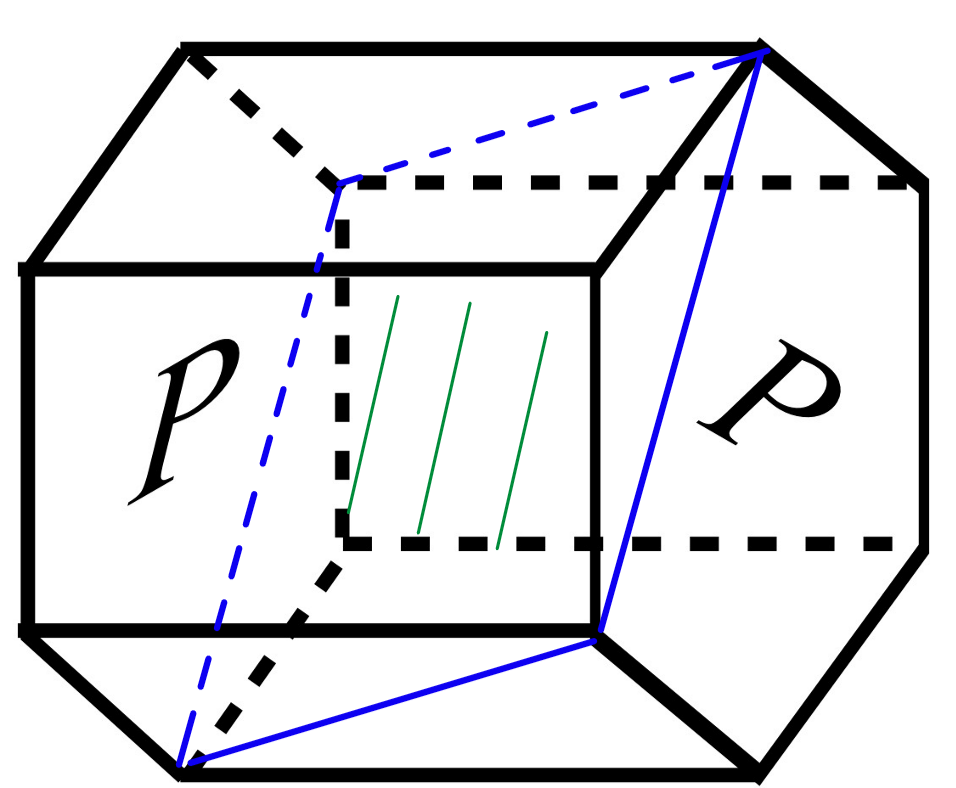}
		\caption{The annulus under $60^{\circ}$ rotation generates a filling surface in $M_4$}
		\label{m4}
		\end{subfigure}
		\hspace{0.1in}
		\begin{subfigure}{0.3\textwidth}
				\centering
		\includegraphics[width=.7\linewidth]{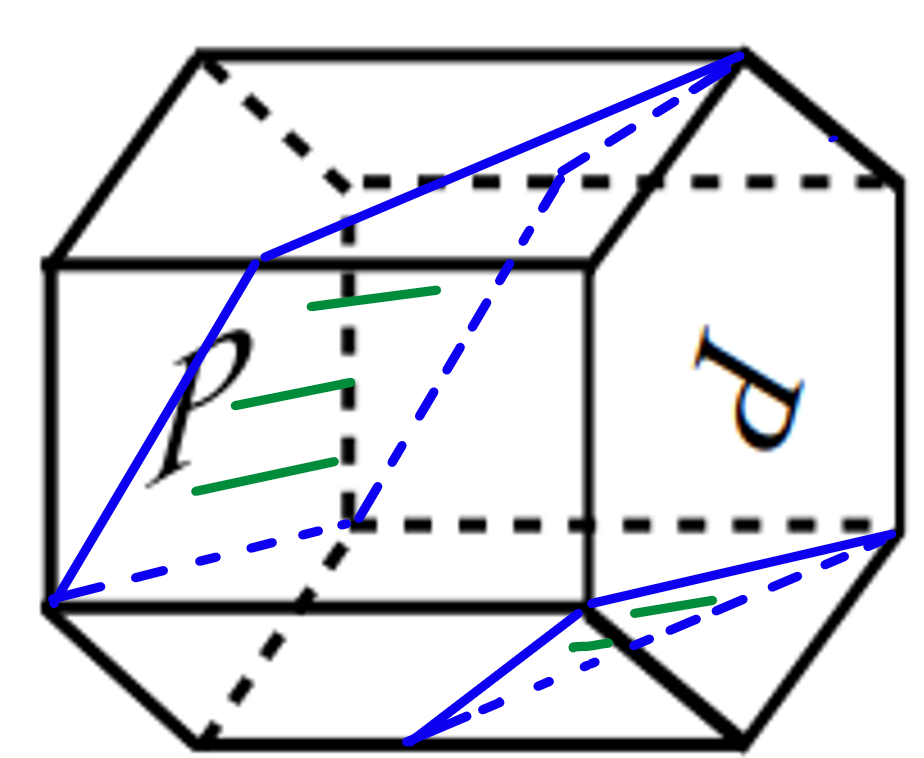}
		\caption{The annulus under $120^{\circ}$ rotation generates a filling surface in $M_5$}
		\label{m5}
		\end{subfigure}
		\caption{Filling surfaces in flat manifolds}
		\label{fillingflat}
	\end{figure}
	
	Similarly, we construct a filling surface in $M_5$ (\Cref{m5}). 
	
The last one is the Hantzsche-Wendt manifold $M_6$, which has as a $4$-fold cover the $3$-torus $M_1$. From the isometries at [\citenum{cmCovering}, page 3], the covering group $C_6$ acting on some torus $S$ (chosen later) results in four copies of $S$ each representing a different homology class in $H_2(\T^3; \Z)$. Denote the projection by $\pi_6:M_1 \rightarrow M_6$. Take $S$ to be a torus with normal vector $(1,1,1)$, then the other the other tori have normals $(-1, -1, 1), (-1, 1, -1)$, and $(1, -1, -1)$, resulting in contractible complementary regions by \Cref{three families to be filling}. Thus $\bar{S}\coloneqq \pi_6(S)$ is a filling surface in $M_6$. 
	
A closed $3$-manifold with	$\hp \times \R$ geometry has a finite cover $S_g \times S^1$. We adapt the construction of a filling surface in $M_4$. For example, consider instead of a hexagon prism an octagon prism $O\times I$ with proper gluings to form $S_2\times S^1$. A quotient manifold $M$ is the prism with $O \times 0$ and $O \times 1$ identified by a translation with a $2\pi/8$-twist ($S_2\times S^1\rightarrow M$ is an $8$-fold cover). 
The annulus formed by gluing two parallelograms in \Cref{fillinghyperbolic} generates a filling surface $S$ in $M$ under the covering group. Now we show $[S]$ is filling. The annulus generates $16$ families of planes in $\hp\times \R$. Let $\tg$ be a vertical fiber perpendicular to the octagon face $O\times 1$. Then $\tg$ meets one of the planes exactly once, implying $\pi_1(M-S)$ does not contain $[\gamma]$. Similarly, if $\tg$ is any other geodesic, then we can find a plane $\tp$ from the $16$ planes so that $\tg$ and $\tp$ intersect exactly once. This implies that $[S]$ is filling. 
\begin{figure}[h]
	\centering
	\includegraphics[scale=0.3]{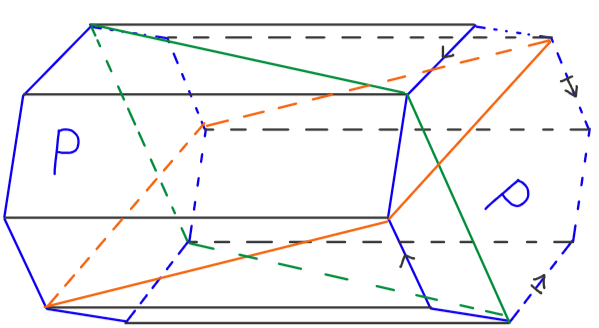}
	\caption{A filling surface in an $\hp\times \R$ manifold	}
	\label{fillinghyperbolic}
\end{figure} 
\end{proof}
 \subsection{The measures and ranks of a filling surface and its ambient manifold}
The proof of \Cref{isoperimetric inequality filling surface has area at least volume}
 \begin{proof}
 	Let $\rk(\cdot)$ denote the rank of a group (the minimum number of generators). 
 	Suppose $[S]$ is filling. By \Cref{filling representative}, a least area representative $S$ cuts $M$ into balls. Moreover, $\area(S)\leq 4\pi(g-1)$ and $\rk(\pi_1(S))=2g$. The boundary components of $M-S$  together form two copies of $S$.
 	By [\citenum{rjFoundationsHyperbolicManifolds}, (3.4.6)], the hyperbolic metric in spherical coordinates is
 	$$\di s^2_{\hy} =\di r^2+\sinh^2 r \, \di s^2_{S^2}, \tn{ where }\di s^2_{S^2} \tn{ is the standard metric on the Euclidean unit sphere.}$$
 	Then the volume of a hyperbolic $3$-ball is 
 	\begin{equation}\label{volume}
 		\vol(B(r))= \area(S^2) \int_0^r\sinh^2 t\, \di t =  \frac{\pi}{2} (e^{2r}-e^{-2r}-4r)
 	\end{equation}
 	while the hyperbolic area 
 	\begin{equation}\label{area of geodesic disk}
 		\area(\partial B(r))= \area(S^2)\sinh^2 r=\pi(e^{2r}+e^{-2r}-2).
 	\end{equation}
 	It is elementary to prove $2e^{-2r}+4r \geq 2$ and hence 
 	$\area(\partial B(r)) \geq 2\vol(B(r)).$
 	Now consider a ball $B(r)$ such that $\vol(B(r))=\vol(M)$. 
 	By classical isoperimetric inequality for topological balls in hyperbolic $3$-space (e.g., [\citenum{cdPresentationLengthVolume}, p.2]), we have 
 	$$\area(\partial(M-S))=2\area(S) \geq \area(\partial B(r)) \geq 2\vol(B(r))=2\vol(M).$$
 	Equality is only achieved at $r=0$ and hence it is a strict inequality.
 	Let $V(r)=\vol (B(r))$ and $\rk(\pi_1(M))$ the rank of $\pi_1(M)$. 
 	[\citenum{bglsCountingArithmeticLattices}, p. 2204] gives an explicit formula relating the rank of a lattice to its covolume, which for $\psltc$ is 
 	$$\rk(\pi_1(M)) \leq \frac{V(1.25\epsilon)}{V(0.25\epsilon)^2}\vol(M)<1.698\times 10^9 \vol(M),$$
 	where $\epsilon=0.104/10$ (we have used Meyerhoff's estimate \cite{mrLowerBoundVolume} of the Margulis constant) and we have used \eqref{volume}. 
 \end{proof}

 \begin{exa}\label{example filling torus in rank tend to infinity}
 	There exists a sequence of $3$-manifolds $M_j$ admitting a filling torus such that $\rank(M_j)\rightarrow \infty$. 
 \end{exa}
 \begin{proof}
 	The construction in \Cref{fillinghyperbolic} is an immersed annulus which under the covering transformations glues to a filling torus in the manifold with an eight-fold cover $S_2 \times S^1$. The same construction applies to the manifold whose $4g$-fold cover is $S_g\times S^1$ and produces a filling torus. 
 \end{proof}

\bibliographystyle{alpha}
\bibliography{Reference.bib}
\address{Center for Mathematics and Interdisciplinary Sciences, Fudan University, Shanghai, 200433, China \\
Shanghai Institute of Mathematics and Interdisciplinary Sciences, Shanghai, 200433, China}
\email{xhh@simis.cn}

\end{document}